\documentclass[11pt]{article}
\usepackage{amsmath}
\usepackage{amsfonts}
\usepackage{amsthm}
\usepackage{graphicx}
\usepackage{enumerate}
\usepackage{subcaption}
\usepackage{fullpage}
\usepackage{float}

\usepackage{color}

\theoremstyle{plain}
\newtheorem{thm}{Theorem}[section]
\newtheorem{lem}[thm]{Lemma}
\newtheorem{cor}[thm]{Corollary}

\theoremstyle{definition}
\newtheorem{defn}[thm]{Definition}
\newtheorem{prob}[thm]{Problem}

\theoremstyle{remark}
\newtheorem{rmk}[thm]{Remark}

\newcommand{\of}[1]{\!\left(#1\right)}
\newcommand{\setZ}{\mathbb{Z}}
\newcommand{\setR}{\mathbb{R}}
\newcommand{\setN}{\mathbb{N}}
\newcommand{\setC}{\mathbb{C}}
\newcommand{\iu}{\mathfrak{i}}

\newcommand{\pardiffby}[2]{\frac{\partial #1}{\partial #2}}
\newcommand{\twovector}[2]{\begin{pmatrix}#1\\#2\end{pmatrix}}
\newcommand{\twomatrix}[4]{\begin{pmatrix}#1&#2\\#3&#4\end{pmatrix}}

\DeclareMathOperator{\id}{\mathrm{Id}}

{\par\color{blue}\textsubscript{#1}}%
{\par}

{\par\color{red}\textsubscript{#1}}%
{\par}

\newcommand{\domain}{\Omega}

\newcommand{\Oi}{{\domain_{\rm int}}}
\newcommand{\Oe}{{\domain_{\rm ext}}}

\newcommand{\bdi}{\Gamma}

\newcommand{\X}{\mathcal X}
\newcommand{\Y}{\mathcal Y}
\newcommand{\Xint}{\mathcal X_{\mathrm{int}}}

\newcommand{\base}{\mathcal B}
\newcommand{\radbase}{\radmod\base}
\newcommand{\bdbase}{\bdmod\base}

\newcommand{\radmod}[1]{\tilde{#1}}
\newcommand{\bdmod}[1]{\hat{#1}}

\newcommand{\sol}{u}
\newcommand{\radsol}{\radmod\sol}

\newcommand{\rhs}{f}
\newcommand{\pot}{p}
\newcommand{\testf}{v}
\newcommand{\radtestf}{\radmod\testf}

\DeclareMathOperator{\calQ}{\mathcal{Q}}

\newcommand{\vecs}[1]{\mathbf{#1}}

\newcommand{\radvar}{\xi}
\newcommand{\bdvar}{\vecs{\bdmod{x}}}
\newcommand{\varx}{\vecs x}

\newcommand{\mbf}{m}
\newcommand{\lbf}{s}
\newcommand{\mbfint}{\mbf_{\mathrm{int}}}
\newcommand{\lbfint}{\lbf_{\mathrm{int}}}
\newcommand{\mbfext}{\mbf_{\mathrm{ext}}}
\newcommand{\lbfext}{\lbf_{\mathrm{ext}}}
\newcommand{\mbfrad}{\radmod\mbf}
\newcommand{\mbfbd}{\bdmod\mbf}
\newcommand{\lbfrad}{\radmod\lbf}
\newcommand{\lbfbd}{\bdmod\lbf}

\begin{document}
\title{Complex scaled infinite elements for exterior Helmholtz problems\footnote{The authors acknowledge support from the Austrian Science Fund (FWF): P26252.}}
\author{Lothar Nannen\footnote{lothar.nannen@tuwien.ac.at} \and 
  Markus Wess\footnote{markus.wess@tuwien.ac.at}}

\date{\today}
\maketitle

\begin{abstract}
    The technique of complex scaling for time harmonic wave type equations relies on
  a complex coordinate stretching to generate exponentially decaying solutions.
  In this work, we use a Galerkin method with ansatz functions with infinite support
  to discretize complex scaled Helmholtz resonance problems. We show that the approximation error 
  of the method decays super algebraically with respect to the number of unknowns in radial direction.
  Numerical examples underline the theoretical findings and show the superior efficiency of our 
  method compared to a standard perfectly matched layer method. 

\end{abstract}

\section{Introduction}
\label{sec:introduction}
Perfectly matched layers (PMLs) are a popular method for treating acoustic resonance and scattering problems in open domains (cf. \cite{Berenger:94, KimPasciak:09, ColMonk} or \cite{HislopSigal,moiseyev:98} for the same method under the name complex scaling). The idea behind this method is the application of a complex coordinate stretching to the unbounded exterior domain to generate exponentially decaying outgoing solutions. Subsequently the exterior domain is truncated to a bounded layer resulting in a bounded computational domain. The resulting problem on the now bounded domain can be discretized using a standard finite element method.
Complex scaling can be applied in various ways: Parallel to the coordinate axes, resulting in so called cartesian scalings (cf. \cite{BramblePasciak:2013}), in radial direction (cf. \cite{ColMonk}) or in normal direction with respect to a convex interface (cf. \cite{LassasSomersalo:98}). In this work we will focus on radial scalings, although the method can be extended to cartesian or normal scalings in a straightforward way. 

PMLs are rather easy to implement in standard finite element codes but have the downside that there are many method parameters to choose: The scaling function, the thickness of the layer, and the finite element discretization of the layer. All these method parameters have to be balanced to ensure efficiency of the method.

In this work we present a method which is also based on complex scaling but omits the truncation of the exterior domain. In contrast to \cite{Bermudezetal:07}, where singular scaling profiles in combination with standard finite elements are used, we use a standard linear scaling profile combined with non-standard basis functions. That way, as in \cite{Bermudezetal:07} we do not introduce a truncation error.

As radial basis functions we choose generalized Laguerre functions leading to the complex scaled infinite elements. They converge super-algebraically, lead to sparse, well-conditioned discretization matrices and are simple to couple to interior problems.
  It turns out, that for homogeneous exterior problems these complex scaled infinite elements are equivalent to the Hardy space infinite elements introduced in \cite{HohageNannen:09}. Moreover, they can be applied easily to resonance problems with inhomogeneous exterior domains.

The remainder of the paper is organized as follows: In Section \ref{sec:problem_setting} we define the problems in question and give a brief explanation of the method of complex scaling. In Section \ref{sec:tp_ext_disc} we explain the used tensor product exterior discretizations. The complex scaled infinite elements are defined in Section \ref{sec:infinite_elements} and their connection to Hardy space infinite elements is explained in Subsection \ref{sec:hsm}.
In Section \ref{sec:convergence} we develop some results concerning the approximation of Hankel functions by our ansatz functions. These results explicitly give us the dependency of the approximation error on the method parameters and thus help us in choosing optimal parameters. 

A section consisting of numerical experiments underlines our theoretical findings. We numerically test our approximation results and compare the performance of the infinite elements to the one of a conventional radial PML. Moreover, we show that the method is also applicable to an example with inhomogeneous exterior.

\section{Problem setting}
\label{sec:problem_setting}
Since we are concerned with the Helmholtz equation on unbounded domains we start by specifying the domains in question. Afterwards we define the Helmholtz scattering and resonance problem and give a short introduction on the technique of complex scaling leading to the weak and discrete formulation of the problem.

For $d\in\{1,2,3\}$ let $\domain\subset\setR^d$ be an unbounded open domain  such 
that $\domain$ can be split into a bounded interior part $\Oi$ an unbounded 
exterior part $\Oe$ and an interface $\bdi$. $\Oi,\Oe,\bdi$ should fulfill the following assumptions:
\begin{enumerate}[(i)]
  \item{$\domain=\Oi\dot\cup\bdi\dot\cup\Oe$,}
  \item{there exists $R>0$, such that $\Oi=\domain\cap B_R(0)$, $\Oe=\Omega\setminus\overline\Oi$ and $\bdi=\left\{\varx\in\domain:\|x\|=R\right\}$, and}
  \item{$\Oe=\left\{\left(1+\frac{\xi}{R}\right)\bdmod{\vecs x}:\bdmod{\vecs x}\in\bdi,\xi\in\setR_{>0}\right\}$.}

\end{enumerate}
  Note, that these conditions imply that for each $\vecs x\in\Oe\cup\bdi$ there exists a unique pair $(\xi,\bdmod{\vecs x})\in\setR_{\geq 0}\times\bdi$, such that
  \begin{align}\label{eq:mapping}\vecs x=\left(1+\frac{\radvar}{R}\right)\bdvar.\end{align}
  For the mapping defined by (\ref{eq:mapping}) we also write $\vecs x \of{\radvar,\bdvar}$ and $\radvar\of\varx,\bdvar\of\varx$ for the inverse mapping.
  In the case $d=1$, we have $\bdvar\in\{-R,R\}$. 
  Figure \ref{fig:example_domain} illustrates a two dimensional example of the setting described above.
\begin{figure}
  \begin{subfigure}{0.45\textwidth}
  \centering
  \includegraphics{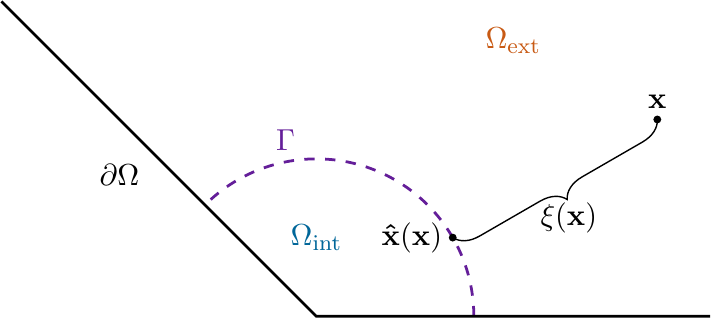}
\end{subfigure}
\begin{subfigure}{0.45\textwidth}
  \centering
  \includegraphics{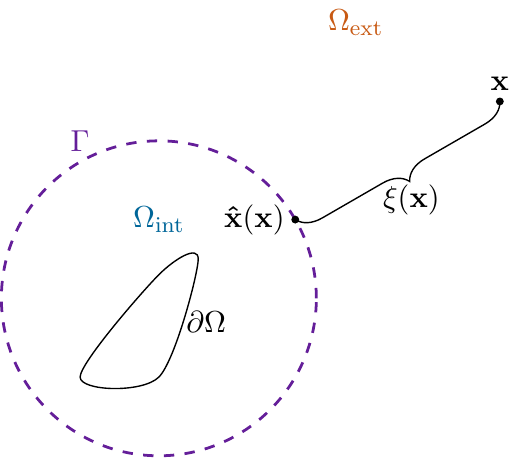}
\end{subfigure}
  \caption{Two dimensional example domains and exterior coordinates}
  \label{fig:example_domain}
\end{figure}
\begin{defn}[Scattering and resonance problem]
  Let $\domain=\Oi\dot\cup\bdi\dot\cup\Oe$ be such that the conditions above hold. Moreover, let $\pot,\rhs\in L_2\of\domain$ such that $p|_\Oe\equiv 1$ and $\mathrm{supp}f\subset\Oi$.
  For a fixed frequency $\omega\in\setC\setminus\{0\}$ we call the problem:
  Find $\sol\in H^2\of\domain$ such that
  \begin{align}
    \label{eq:resonance_prob}
    -\Delta \sol\of{\vecs{x}}-\omega^2\pot\of{\vecs x}\sol\of{\vecs x}&=\rhs\of{\vecs x},&\vecs x&\in\domain,\\
    \nonumber
    \sol&\text{  fulfills some b.c.,}&\vecs x&\in\partial\Omega,\\
    \nonumber
    \sol&\text{  is outgoing,}&\|\vecs x\|&\to\infty,
  \end{align}
  the \emph{Helmholtz scattering problem}.
  The problem: Find $\omega\in \setC^+:=\{z\in\setC:\Re\of\omega\geq 0\},\,\sol\in H^2\of\domain\setminus\{0\}$, such that
  \begin{align}
    \label{eq:scattering_prob}
    -\Delta \sol\of{\vecs{x}}&=\omega^2\pot\of{\vecs x}\sol \of{\vecs x},&\vecs x&\in\domain,\\
    \nonumber
    \sol&\text{  fulfills some b.c.,}&\vecs x&\in\partial\Omega,\\
    \nonumber
    \sol&\text{  is outgoing,}&\|\vecs x\|&\to\infty,
  \end{align}
  is called the \emph{Helmholtz resonance problem}.
\end{defn}
In the following we will focus on the resonance problem.
\subsection{Radiation condition}
\label{sec:radiation_condition}
We call a solution $\sol$ of (\ref{eq:scattering_prob}) or (\ref{eq:resonance_prob}) \emph{outgoing} if it can be written in $\Oe$ (i.e. for all $\radvar\in\setR_{\geq 0},\bdvar\in\bdi$) as 
\begin{align}
  \label{eq:radiation_cond}
  \sol(\vecs x\of{\radvar,\bdvar})=\begin{cases}
    \exp(i\omega\vecs x\of{\radvar,\bdvar}),&d=1,\\
    \sum_{\nu=-\infty}^\infty\alpha_\nu H^{(1)}_{|\nu|}\of{\omega(R+\radvar)}\Phi_\nu\of{\frac{1}{R}\bdvar},&d=2,\\
    \sum_{\nu=0}^\infty\sum_{j=0}^{m_j}\beta_{\nu,j} h^{(1)}_\nu\of{\omega(R+\radvar)}Y_{\nu,j}\of{\frac{1}{R}\bdvar},&d=3,
  \end{cases}
\end{align}
where $H_\nu^{(1)}$ are the Hankel functions of the first kind, $h_\nu^{(1)}$ the spherical Hankel functions of the first kind, $\Phi_\nu$ the cylindrical harmonics and $Y_{\nu,j}$ the spherical harmonics. For the definition of these functions and details to the radiation condition see e.g. \cite{ColtonKress:98}. Note, that in \cite[Chapter 2]{ColtonKress:98} only positive frequencies are considered. Nevertheless, since the functions $H_\nu^{(1)},\, h_\nu^{(1)}$ are analytic for arguments with positive real part, we can use this radiation condition for complex frequencies with positive real part as well. 

An equivalent formulation of this radiation condition can be derived using boundary integral representations (see \cite{SteinbachUnger:2012}). Both formulations imply that an outgoing solution has an analytic continuation to $\vecs x\of{\mathbb C,\bdi}$. In the following we will use the symbol $u$ for the analytic continuation as well. 
%  \begin{rmk}
%    \replaced{Due to the fact that the functions $H_\nu^{(1)},\, h_\nu^{(1)},\,\Phi_\nu,\,Y_{\nu,j}$  are analytic, this }{This also }implies that an outgoing solution has an analytic continuation to $\vecs x\of{\mathbb C,\bdi}$. In the following we will use the symbol $u$ for the analytic continuation as well.
%  \end{rmk}

\subsection{Complex scaling}
\label{sec:complex_scaling}
To incorporate (\ref{eq:radiation_cond}) into our problem we use the technique of complex scaling.
In this work we only consider linear complex scalings of the form
\begin{align}
\tau(\radvar)&:=\sigma\radvar,\nonumber\\
  \label{eq:scaling}
\gamma\of{\vecs x\of{\radvar,\bdvar}}&:=\begin{cases}
  \vecs x,&\vecs x\in\Oi,\\
  \vecs x\of{\tau\of\radvar,\bdvar},&\vecs x\in\Oe,
\end{cases}
\end{align}
for a given $\sigma\in\mathbb C$ with $\Im\of\sigma>0$. We denote the Jacobian of the scaling by
$$J_\sigma\of{\vecs x}=J_\sigma\of{(x_1,\ldots,x_d)^T}:=\left(\pardiffby{\gamma_i\of{(x_1,\ldots,x_d)^T}}{x_j}\right)_{i=1,\ldots,d,\, j=1,\ldots,d}.$$
Due to (\ref{eq:radiation_cond}) and the fact that the (spherical) Hankel functions $h^{(1)}_\nu,\,H^{(1)}_\nu$ behave like $\radvar\mapsto\exp\of{i\radvar}\alpha_\nu\of\radvar$ for certain rational functions $\alpha_\nu$ and $\radvar\to\infty$ (cf. Definition \ref{defn:hankel} and \cite{ColtonKress:98}), this gives for $\sol$ of the form (\ref{eq:radiation_cond})\begin{align*}
  \lim_{\radvar\to\infty}\sol\of{\gamma\of{\varx\of{\radvar,\bdvar}}}=0,
\end{align*}
for scalings of the form (\ref{eq:scaling}) and frequencies $\omega$ with $\Im\of{\sigma\omega}>0$.

\subsection{Weak formulation}
\label{sec:weak_form}
Since the complex scaled solution $u\circ\gamma$ decays exponentially for $\|\varx\|\to\infty$, it is also square integrable and we can state a weak formulation of (\ref{eq:resonance_prob}) using the following bilinear forms:

\begin{defn}
  For $f,g\in H^1\of{\Oe}$ we define
  \label{defn:int_ext_bf}
  \begin{align*}
    \mbfint\of{f,g}&:=\int_\Oi \pot\of\varx f\of\varx g\of\varx\,d\varx,\\
    \lbfint\of{f,g}&:=\int_\Oi \nabla f\of\varx \cdot\nabla g\of\varx\,d\varx,\\
    \mbfext^\sigma\of{f,g}&:=\int_\Oe f\of\varx g\of\varx\det J_\sigma\of\varx\,d\varx,\\
    \lbfext^\sigma\of{f,g}&:=\int_\Oe\left(J_\sigma\of{\vecs x}^{-T}\nabla f\of{\vecs x}\right)\cdot\left(J_\sigma\of{\vecs x}^{-T}\nabla g\of{\vecs x}\right)\det J_\sigma\of\varx\,d\varx.
  \end{align*}
\end{defn}

%\begin{prob}
%  \label{prob:weak_scattering}
%  Let $\rhs,\pot\in L_2\of\domain$ with support only in $\Oi$. Find $\sol\in H^1\of\domain$, such that
%  \begin{multline}
%    \label{eq:weak_scattering}
%    \int_\domain \left(J\of{\vecs x}^{-1}\nabla\sol\of{\vecs x}\right)\cdot\left(J\of{\vecs x}^{-1}\nabla \testf\of{\vecs x}\right)\det J\of{\vecs x}\,d\vecs x\\
%    -\omega^2\int_{\domain}\pot\of{\vecs x}\sol\of{\vecs x}\testf\of{\vecs x}\det J\of{\vecs x}\,d\vecs x\\
%    =\int_\domain \rhs\of{\vecs x}\testf\of{\vecs x}\det J\of{\vecs x}\,d\vecs x.
%  \end{multline}
%  for all $\testf\in H^1\of\domain$.
%\end{prob}

\begin{prob}
  \label{prob:weak_resonance}
  Find $\sol\in H^1\of\domain\setminus\{0\}$, $\omega\in\setC^+$, such that
  \begin{align}
    \label{eq:weak_resonance}
    \lbfint\of{\sol,\testf}+\lbfext^\sigma\of{\sol,\testf}=\omega^2\left(\mbfint\of{\sol,\testf}+\mbfext^\sigma\of{\sol,\testf}\right),
  \end{align}
  for all $\testf\in H^1\of\domain$.
\end{prob}
\begin{rmk}
  The weakly formulated Problem \ref{prob:weak_resonance} assumes homogeneous Neumann boundary conditions on $\partial\domain$. For Dirichlet or mixed boundary conditions the problem has to be adapted accordingly.
\end{rmk}
\subsection{Discrete formulation}
Our goal is to discretize Problem \ref{prob:weak_resonance}. To this end we pick $\mathcal N\in\setN$ and a family of functions
$\base_{\mathcal N}:=\left\{b_0,\ldots,b_{\mathcal N}\right\}\subset H^1\of\domain$ and define the discrete space $\X_\mathcal N$ by
$$\X_\mathcal N:=\mathrm{span}\of{\base_{\mathcal N}}\subset H^1\of\domain.$$
Defining the mass- and stiffness matrix by
\begin{equation}
  \label{eq:mass_stiffness}
  \vecs M:=(m_{i,j})_{i,j=0,\ldots,\mathcal N},\quad\vecs S:=(s_{i,j})_{i,j=0,\ldots,\mathcal N}
\end{equation}
and
\begin{equation}
  \label{eq:mass_stiffness1}
m_{i,j}=\mbfint\of{b_i,b_j}+\mbfext^\sigma\of{b_i,b_j},\quad s_{i,j}=\lbfint\of{b_i,b_j}+\lbfext^\sigma\of{b_i,b_j}
\end{equation}
respectively, we can formulate the discrete problem by
\begin{prob}
  \label{prob:disc}
  Find $(\omega,\vecs u)\in\setC^+\times\setC^{\mathcal N}\setminus\{0\}$, such that
  $$\vecs S \vecs u=\omega^2\vecs M\vecs u.$$
\end{prob}
The discrete Problem \ref{prob:disc} can be solved using standard eigenvalue solvers (see e.g. \cite{Saad}).
In the following our task will be to find a suitable basis $\base_{\mathcal N}$.

\section{The exterior problem}
\label{sec:tp_ext_disc}

In this section we will exploit the inherent structure of the exterior domain to find a simple way of discretizing it without having to mesh it explicitly.
To simplify the notation we will focus on the case $d=3$ only. 

\subsection{Exterior variational formulation}
For the remainder of this section we will assume that $\varphi:M\to\bdi$ is a diffeomorphism for some open set $M\subset\setR^2$.
For the case $\bdi=\{\varx\in\setR^3:\|\varx\|=R\}$ an example for $\varphi$ is given by the usual spherical coordinates
    \begin{align*}
      \varphi:\begin{cases}
        [0,2\pi)\times[0,\pi)&\to\bdi,\\
        (\theta,\phi)&\mapsto R\begin{pmatrix}\sin\of\theta\cos\of\phi\\\sin\of\theta\sin\of\phi\\\cos\of\theta\end{pmatrix}.
      \end{cases}
    \end{align*}
\begin{lem}
  \label{lem:trafo}
  We can calculate the Jacobian of the coordinate transformation 
  \begin{align*}
    \Psi_\varphi:\begin{cases}
      \setR_{\geq 0}\times M&\to\Oe\cup\bdi,\\
      (\radvar,\eta)&\mapsto \left(1+\frac{\radvar}{R}\right)\varphi\of\eta,
    \end{cases}
  \end{align*}
    its inverse, and its determinant by
  \begin{align*}
    D\Psi_\varphi\of{\radvar,\eta}&=\left(\frac{1}{R}\varphi\of\eta,\left(1+\frac{\radvar}{R}\right)D\varphi\of\eta\right),\\
    \left(D\Psi_\varphi\of{\radvar,\eta}\right)^{-1}&=\twovector{\frac{1}{R}\varphi\of\eta^T}{\frac{1}{1+\frac{\radvar}{R}}D\varphi\of\eta^\dagger},\\
    |\det D\Psi_\varphi\of{\radvar,\eta}|&=\left(1+\frac{\radvar}{R}\right)^2\sqrt{\left|\det\left( D\varphi\of\eta^TD\varphi\of\eta\right)\right|},
  \end{align*}
  where $A^\dagger:=\left(A^TA\right)^{-1}A^T$ is the pseudo inverse of a matrix $A\in\setC^{3\times 2}$ with full rank.
\end{lem}
\begin{proof}
The Jacobian can be obtained by straightforward differentiation. 
Its inverse can be easily verified using the facts that $\varphi\of\eta^TD\varphi\of\eta=0$ and $D\varphi\of\eta^\dagger D\varphi\of\eta = I$.

  For obtaining the determinant we calculate
  \begin{align*}
    \det\of{ D\Psi_\varphi^T D\Psi_\varphi}&=\det\of{\twovector{\frac{1}{R}\varphi\of\eta^T}{\left(1+\frac{\radvar}{R}\right)D\varphi\of\eta^T}\left(\frac{1}{R}\varphi\of\eta,\left(1+\frac{\radvar}{R}\right)D\varphi\of\eta\right)}\\
    &=\det\twomatrix{1}{0}{0}{\left(1+\frac{\radvar}{R}\right)^2D\varphi\of\eta^T D\varphi\of\eta}\\
    &=\left(1+\frac{\radvar}{R}\right)^4\det\of{{D\varphi\of\eta^T D\varphi\of\eta}}.
  \end{align*}
  By taking the square root we obtain the desired result.
\end{proof}
%\begin{lem}
%  \label{lem:polar_trafo}
%  We define the surface gradient on $\bdi$ by
%$$\hat \nabla:=\left(I-\frac{1}{R^2}\bdvar\bdvar^T\right)\nabla.$$
%Then the gradient can be rewritten in polar coordinates by
%  \begin{align*}
%    \nabla =\frac{\bdvar}{R}\pardiffby{}{\radvar}+\frac{1}{1+\frac{\radvar}{R}}\hat\nabla.
%  \end{align*}
%  Moreover we can perform integral transformations for $f\in\L_2\of{\Oe}$ in polar coordinates using
%  \begin{align}
%    \int_\Oe f\of{\vecs x}\,d\vecs x =\int_{\setR_{\geq 0}\times\bdi}f\of{\vecs x\of{\radvar,\bdvar}}\left(1+\frac{\radvar}{R}\right)^{2}\,d(\radvar,\bdvar).
%  \end{align}
%\end{lem}
%\begin{proof}
%  TODO
%\end{proof}
\begin{defn}
  Let $M\subset\setR^2$ and $\varphi:M\to\bdi$ be a diffeomorphism. Then we define the surface gradient on of a function $f:\bdi\to\setC$ by
  $$\hat \nabla f\of{\varphi\of\eta}:=\left(D\varphi\of\eta^\dagger\right)^T\nabla_\eta\of{f\circ\varphi}\of\eta.$$
\end{defn}
  It can be shown that the surface gradient $\hat\nabla$ defined above is independent of the specific embedding $\varphi$.
\begin{thm}
  \label{thm:weak_polar}
  Let $f,g\in H^1\of\Oe$ and $\breve f\of{\radvar,\bdvar}:=f\of{\varx\of{\radvar,\bdvar}}$, $\breve g\of{\radvar,\bdvar}:=g\of{\varx\of{\radvar,\bdvar}}$.
Then the exterior bilinear forms from Definition \ref{defn:int_ext_bf} can be rewritten in the coordinates $\radvar,\bdvar$ by
\begin{align*}
  m^\sigma_{\mathrm{ext}}\left(f,g\right)&=\sigma\int_{\setR_{\geq 0}\times\bdi}\!\!\!\!\!\! \breve f\of{\radvar,\bdvar}\breve g\of{\radvar,\bdvar}\left(1+\frac{\sigma\radvar}{R}\right)^{2}\,d(\radvar,\bdvar),\\
  s^\sigma_{\mathrm{ext}}\left(f,g\right)&=\frac{1}{\sigma}\int_{\setR_{\geq 0}\times\bdi} \pardiffby{\breve f}{\radvar}\of{\radvar,\bdvar}\pardiffby{\breve g}{\radvar}\of{\radvar,\bdvar}\left(1+\frac{\sigma\radvar}{R}\right)^{2}\,d(\radvar,\bdvar),\\
  &+\sigma\int_{\setR_{\geq 0}\times\bdi} \hat\nabla \breve f\of{\radvar,\bdvar}\hat\nabla \breve g\of{\radvar,\bdvar}\,d(\radvar,\bdvar),
\end{align*}  
where integration over $\bdi$ of a function $h:\bdi\to\setC$ means integration by the surface measure i.e. 
$$\int_{\varphi\of M}h\of\bdvar\,d\bdvar:=\int_M h\of{\varphi\of\eta}\sqrt{\left|\det\of{D\varphi\of\eta^T D\varphi\of\eta}\right|}\,d\eta.$$
\end{thm}
\begin{proof}
  Using the determinant calculated in Lemma \ref{lem:trafo} and the fact that
  $$\gamma\of{\Psi_\varphi\of{\radvar,\eta}}=\Psi_\varphi\of{\sigma\radvar,\eta},$$
  we obtain
  \begin{align*}
    D\of{\gamma\circ\Psi_\varphi}\of{\radvar,\eta}&=D\Psi_\varphi\of{\sigma\radvar,\eta}\twomatrix{\sigma}{0}{0}{\id_2},\\
    \left(D\of{\gamma\circ\Psi_\varphi}\of{\radvar,\eta}\right)^{-1}&=\twomatrix{\frac{1}{\sigma}}{0}{0}{\id_2}\left(D\Psi_\varphi\of{\sigma\radvar,\eta}\right)^{-1},\\
    \det D\of{\gamma\circ\Psi_\varphi}\of{\radvar,\eta}&=\sigma\det D\Psi_\varphi\of{\sigma\radvar,\eta}.
  \end{align*}
After applying the transformation rule we immediately obtain the formula for $\mbfext^\sigma$.

For the formula for $\lbfext^\sigma$ we calculate
\begin{align*}
  \nabla f\of{\Psi_\varphi\of{\radvar,\eta}}&=\left(D\of{\gamma\circ\Psi_\varphi}\of{\radvar,\eta}\right)^{-T}\nabla_{\radvar,\eta}\left(f\circ\Psi_\varphi\right)\of{\radvar,\eta}\\
  &=\left(\frac{1}{\sigma R}\varphi\of\eta,\frac{1}{1+\frac{\radvar\sigma}{R}}\left(D\varphi\of\eta^\dagger\right)^T\right)\twovector{\pardiffby{f\circ\Psi_\varphi}{\radvar}\of{\radvar,\eta}}{\nabla_\eta\of{f\circ\Psi_\varphi}\of{\radvar,\eta}}\\
  &=\frac{1}{\sigma R}\varphi\of\eta\pardiffby{\breve f}{\radvar}\of{\radvar,\varphi\of\eta}+\frac{1}{1+\frac{\sigma\radvar}{R}}\hat\nabla\breve f\of{\radvar,\varphi\of\eta}.
\end{align*}
Plugging this into the integral and applying the transformation rule leads to the desired result.
\end{proof}
\subsection{Tensor product discretization of the exterior problem}
Let
\begin{align*}
\radbase_N&:=\{\phi_n:n=0,\ldots,N\}\subset H^1\of{\setR_{\geq 0}},\\
\bdbase_M&:=\{b_j:j=0,\ldots,M\}\subset H^1\of\bdi,
\end{align*}
be families of linearly independent functions.
Then we define
discrete spaces on $\setR_{\geq 0}$ and $\bdi$ respectively by
$$\radmod\X_N:=\mathrm{span}\of{\radbase_N}\subset H^1\of{\setR_{\geq 0}},$$
and 
$$\bdmod\X_M:=\mathrm{span}\of{\bdbase_M}\subset H^1\of\bdi.$$
To discretize the exterior problem, we use a tensor product space of the form
$$\radmod\X_N\otimes\bdmod\X_M:=\mathrm{span}\{\phi\otimes b:(\radvar,\bdvar)\mapsto\phi\of\radvar b\of\bdvar:\phi\in\radbase_N,b\in\bdbase_M\}.$$

To obtain the entries of the mass and stiffness matrix defined in (\ref{eq:mass_stiffness}) and (\ref{eq:mass_stiffness1}), we need to evaluate the exterior bilinear forms for all pairs of basis functions. Since our basis functions are composed of a radial and a tangential part, we can decompose the bilinear forms accordingly and obtain for $\tilde f,\tilde g\in H^1\of{\setR_{\geq 0}}$ and $\hat f,\hat g \in H^1\of\bdi$
\begin{align*}
  \mbfext^\sigma\left(\tilde f\otimes\hat f,\tilde g\otimes\hat g\right)&=\mbfrad^\sigma_1\of{\tilde f,\tilde  g}\mbfbd\of{\hat f ,\hat g},\\
  \lbfext^\sigma\left(\tilde f\otimes\hat f,\tilde g\otimes\hat g\right)&=\lbfrad^\sigma\of{\tilde f,\tilde g}\mbfbd\of{\hat f,\hat g}+\mbfrad_0^\sigma\of{\tilde f,\tilde g}\lbfbd\of{\hat f,\hat g},
\end{align*}
with
\begin{align*}
  \mbfrad^\sigma_0\of{\tilde f,\tilde g}&=\sigma\int_0^\infty \tilde f\of\radvar\tilde g\of\radvar\,d\radvar,\\
  \mbfrad_1^\sigma\of{\tilde f,\tilde g}&=\sigma\int_0^\infty\left(1+\frac{\sigma\radvar}{R}\right)^2 \tilde f\of\radvar\tilde g\of\radvar\,d\radvar,\\
  \lbfrad^\sigma\of{\tilde f,\tilde g}&=\frac{1}{\sigma}\int_0^\infty\left(1+\frac{\sigma\radvar}{R}\right)^2 \tilde f'\of\radvar\tilde g'\of\radvar\,d\radvar,\\
  \mbfbd\of{\hat f,\hat g}&=\int_\bdi \hat f\of\bdvar\hat g\of\bdvar\,d\bdvar,\\
  \lbfbd\of{\hat f,\hat g}&=\int_\bdi \bdmod\nabla\hat f\of\bdvar\cdot\bdmod\nabla\hat g\of\bdvar\,d\bdvar.
\end{align*}
A usual perfectly matched layer (PML) approach in this tensor product setting would be to truncate the set $\setR_{\geq 0}$ to some finite interval $[0,T]$ for $T>0$ and to use 
$$\radmod\X_N\subset H_0^1\of{[0,T]}:=\{f\in H^1\of{\setR_{\geq 0}}:f(x)=0, x\geq T\},$$
where $\radmod\X_N$ is a finite element space.
Differing from this approach, we will choose basis functions with infinite support to omit truncation and ensure faster convergence. 
Our requirements to the basis functions $\phi_n$ and the discrete space $\radmod\X_N$ are:
\begin{enumerate}[(R 1)]
  \item{The basis functions $\phi_n$ should be easy to evaluate numerically stable,}
    \label{enum:rfirst}
    \label{enum:stable}
  \item{the radial part of the solution should be well approximated by functions from $\tilde\X_N$,}
    \label{enum:approx}
  \item{it should be easy to couple the interior to the exterior problem,}
    \label{enum:coupling}
  \item{the integrals $\int_{\setR_{\geq 0}}p\of\radvar\phi_n\of\radvar\phi_j\of\radvar\,d\radvar$ and$\int_{\setR_{\geq 0}}p\of\radvar\phi_n'\of\radvar\phi_j'\of\radvar\,d\radvar$  should be easy to compute (numerically), for polynomials $p$,}
    \label{enum:numeric}
  \item{the discretization matrices should be sparse, and}
    \label{enum:sparse}
  \item{the condition numbers of the discretization matrices should behave well for large values of $N$.}
    \label{enum:cond}
    \label{enum:rlast}
\end{enumerate}

\section{Infinite elements based on complex scaling}
\label{sec:infinite_elements}
Complex scaling leads to anisotropic solutions. In the interior domain as well as in the tangential direction of the exterior domain the oscillating behaviour of the function dominates. In radial direction of the exterior domain the exponential decay is crucial. Therefore, in order to reduce computational costs it is natural to choose suitable basis functions for the different parts of the solution.

\subsection{Interior and interface discretization}
For discretizing the interior problem basically any discrete space $\Xint=\mathrm{span}\{b_j:j=0,\ldots,L\}\subset H^1\of{\Oi}$ such that $\Xint|_\bdi:=\{f|_\bdi:f\in\Xint\}\subset H^1\of\bdi$ can be used. The trace space of this interior discrete space is then used for the interface discretization (cf. Section \ref{sec:tp_ext_disc}), i.e.
$$\bdmod\X_M:=\Xint|_\bdi=\mathrm{span}\{b_j|_\bdi:j=0,\ldots,L\}\subset H^1\of\bdi.$$
%\begin{rmk}
  In our examples we will choose $\Xint$ as a standard high order conforming finite element space.
  Since in this case all of the basis functions corresponding to inner nodes in $\Oi$ will be zero on the interface $\bdi$, we expect the dimension of $\bdmod\X_M$ to be much smaller than the dimension of $\Xint$.
%\end{rmk}
%

\subsection{Radial discretization}
For the radial discretization we use the space of generalized Laguerre functions. These functions are used as basis functions of spectral methods for equations on unbounded domains with exponentially decreasing solutions (cf. \cite[Section 7.4]{ShenTangWang}). We will see in the following, that they are a suitable choice considering our requirements (R \ref{enum:rfirst})-(R \ref{enum:rlast}).
  Following \cite[Section 7.1]{ShenTangWang}, we define the generalized Laguerre polynomials as follows:
\begin{defn}
  \label{defn:laguerre}
  For $n\in\setN_0$ and $m\in\setZ$, we define the generalized Laguerre polynomials by
  \begin{align*}
    L_{n,m}(x):=\sum_{k=0}^n\begin{pmatrix}n+m\\n-k\end{pmatrix}\frac{(-x)^k}{k!}.
  \end{align*}
  Further we define the generalized Laguerre functions by
  \begin{align}
    \phi_{n,m}(x):=\exp\of{-x}L_{n,m}(2x).
  \end{align}
  We will shorten the notation by writing $\phi_n:=\phi_{n,0}$ and $L_n:=L_{n,0}$.
  Moreover we define the radial discrete space by
  $$\radmod\X_N:=\mathrm{span}\left\{\phi_n:n=0,\ldots,N\right\}.$$
\end{defn}
We proceed to study whether the basis functions defined in Definition \ref{defn:laguerre} satisfy our requirements (R \ref{enum:rfirst})-(R \ref{enum:rlast}). To this end, we state a few properties of the generalized Laguerre functions.
\begin{lem}[properties of the generalized Laguerre functions]
  \label{lem:prop_laguerre}
  
  \begin{enumerate}[(i)]
 \leavevmode
  \item{For $k,n,j\in\setN_0$, $m\in\setZ$ the functions $\phi_{n,m}\in H^j\of{\setR_{\geq 0}}$ and $$\left(\phi_n,\phi_k\right)_{L_2\of{\setR_{\geq 0}}}=\frac{1}{2}\delta_{n,k}.$$ The functions $\left\{\frac{1}{\sqrt 2}\phi_k,k\in\setN_0\right\}$ form a complete orthonormal system of $L_2\of{\setR_{\geq 0}}$.}
    \label{enum:laguerre_ortho}
  \item{For $n\in\setN,m\in\setZ$,  $$\phi_{n,m-1}=\phi_{n,m}-\phi_{n-1,m}.$$}
    \label{enum:laguerre_indrec}
  \item{For $k\in\setN_0$, $x\in\setC$ $$\frac{d^k}{dx^k}\phi_{n,m}\of x=(-1)^k\phi_{n-k,m+k}\of x.$$}
    \label{enum:laguerre_diff}
  \item{For $n\in\setN_0,m\in\setZ$, $x\in\setC$ $$L_{n,m}(x)=\frac{\exp(x)}{x^{m}n!}\frac{d^n}{dx^n}\of{\exp(-x)x^{n+m}}.$$}
    \label{enum:laguerre_rodriguez}
  \item{For $n,k,l\in\setN$, $p\in\Pi_n$ and $|l-k|>n$ $$\left(p\phi_l,\phi_k\right)_{L_2\of{\setR_{\geq 0}}}=0.$$}
    \label{enum:laguerre_sparse}
  \item{For $m\in\setZ$, $N\in\setN$, $$\radmod\X_N=\mathrm{span}\left\{\phi_{n,m}:n=0,\ldots,N\right\}$$}
    \label{enum:laguerre_space}
  \item{For $j\in\setN_0$ $$\quad \phi_j(0)=\delta_{0,j}.$$}
    \label{enum:laguerre_coupling}
  \item{For $k\in\setN_{\geq 2}$, $x\in\setC$ \begin{equation}\label{eq:laguerre_rec}k\phi_k(x)=(2k-1-2x)\phi_{k-1}(x)-(k-1)\phi_{k-2}(x).\end{equation}}
    \label{enum:laguerre_rec}
    \item{For $t,x\in\setC$, $|t|<1$ 
        $$\sum_{k=0}^\infty L_k(x)t^k=\frac{\exp\of{-\frac{tx}{1-t}}}{1-t}.$$}
    \label{enum:laguerre_gen}
  \end{enumerate}
\end{lem}
\begin{proof}
  All of the statements are easily checked by the reader and can be found e.g. in \cite[Chapter~22]{Abramowitz}
\end{proof}
  Item (\ref{enum:laguerre_coupling}) of Lemma \ref{lem:prop_laguerre} shows, that only the first radial basis function has to be coupled to an interior basis function i.e. (R \ref{enum:coupling}) is fulfilled. Moreover items (\ref{enum:laguerre_ortho}) and (\ref{enum:laguerre_sparse}) together with (\ref{enum:laguerre_diff}) and (\ref{enum:laguerre_indrec}) tell us that the resulting matrices will be sparse ((R \ref{enum:sparse})).

\subsection{Coupling the interior and the exterior problem}
Since we want to create a conforming discrete space for the whole problem, we need to couple our interior and exterior discrete spaces in a manner such that the resulting space is equivalent to a subspace of $H^1\of\Omega$. We achieve this by using
$$\Y:=\left\{(u_{\mathrm{int}},u_{\mathrm{ext}}):u_{\mathrm{int}}\in\Xint,u_{\mathrm{ext}}\in\X_{\mathcal N}, u_\mathrm{int}|_\bdi = u_\mathrm{ext}(0,\cdot)\right\}.$$
With an embedding defined by
$$\iota:\begin{cases}
  \Y&\to H^1\of\domain,\\
  \iota\of{(u,v)}\of\varx&:=\begin{cases}u\of\varx,&\varx\in\Oi,\\ v\of{\radvar\of\varx,\bdvar\of\varx},&\varx\in\Oe,\end{cases}
\end{cases}$$
we have
$$\iota\of{\Y}\subset H^1\of\domain.$$

To obtain a basis of $\Y$ we have to couple the basis functions, such that the resulting functions are continuous. This can be done by identifying an interior basis function $b_j$ with non-vanishing trace on $\Gamma$ with the exterior basis function $\phi_0\otimes b_j|_\bdi$. Note, that due to the tensor product structure of the exterior space the parts of $\vecs S$ and $\vecs M$ that correspond to the exterior domain can be assembled by calculating the radial and interface part separately and tensorizing them appropriately.
%\begin{rmk}
%  Due to the tensor product structure of the exterior space the parts of $\vecs S$ and $\vecs M$ that correspond to the exterior domain can be assembled by calculating the radial and interface part seperately and tensorizing them appropriately.
%\end{rmk}

\subsection{Stable evaluation and numerical integration}
The generalized Laguerre functions can be evaluated numerically stable by using the recursion given in Lemma \ref{lem:prop_laguerre}(\ref{enum:laguerre_rec}).
We use Gauss rules for $(0,\infty)$ with weighting function $\exp(-\cdot)$ to obtain exact quadrature rules for the Laguerre functions (see \cite[Chapter~7.1.2]{ShenTangWang}).

\begin{rmk}
  This enables us to also deal with inhomogeneous potentials in the exterior domain which is not possible in a straightforward way using classical Hardy Space infinite elements.
\end{rmk}

\subsection{Comparison to the Hardy Space infinite element method}
\label{sec:hsm}
The Hardy space infinite element method introduced in \cite{HohageNannen:09} uses the so called pole condition \cite{SD:95,Schmidt:02,PC1} as radiation condition. In its standard form, this pole condition is equivalent to the radiation condition underlying the complex scaling, which is equivalent to the radiation condition of Sec.~\ref{sec:radiation_condition} for certain domains of complex frequencies including positive frequencies (see \cite{PC2} or more explicitly for waveguides in \cite{HohageNannen:15,HallaNannen:18}).

The Hardy space infinite element method is a tensor product method as introduced in Sec.~\ref{sec:tp_ext_disc}. But since the pole condition characterizes radiating solutions of the Helmholtz equation by the poles or singularities of their Laplace transform, the discretization in the radial direction is done for the Laplace transformed function. The basis functions are elements of certain Hardy spaces such that they satisfy the pole condition. In order to use these basis functions, the Helmholtz equation has to be transformed into the Laplace domain leading to quite unusual variational formulations in unusual Hilbert spaces. Nevertheless, it is a pure Galerkin method.

For a comparison with the complex scaled infinite elements of this paper, Section 4.2 of \cite{HohageNannen:09} is of importance. In this section, the Hardy space variational formulation is related to a complex scaled variational formulation via a Fourier transform. If the isomorphism $\calQ$ defined there is applied to the Hardy space basis functions from \cite[Sec.~2.4]{HohageNannen:09}, we arrive at the generalized Laguerre functions of the preceding subsections. Hence, the discretization matrices of the Hardy space infinite element method are exactly the same as those of the complex scaled infinite elements.

For the Helmholtz equation with homogeneous exterior domain the complex scaled infinite element method is therefore exactly identical to the standard Hardy space infinite element method. Only the functional setting and the theoretical justification is different. There are two situations, where the two methods differ. If the exterior domain is inhomogeneous with coefficient functions depending on the radius, the Hardy space infinite element method is complicated to use due to the involved Laplace transform. Nevertheless, inhomogeneous exterior problems with dependencies only on the surface variable can be solved with the pole condition framework as well (see \cite{NannenSchaedle:09}). On the other hand, the two pole Hardy space method introduced in \cite{Hallaetal:16} uses a more complicated form of the pole condition, which is not equivalent to a standard complex scaling radiation condition. So e.g. for elastic waveguide problems with different signs of group and phase velocity, the Hardy space infinite element method of \cite{HallaNannen:15} cannot be reinterpreted directly as a complex scaled infinite element method.

For problems, where the two methods are essentially identical, the convergence results in \cite{HohageNannen:09,Halla:16} can be used for complex scaled infinite elements as well. Nevertheless, in the following section we present more detailed approximation results for the infinite element method, which have not been derived so far. They may help choosing appropriate method parameters in practice.

\section{Approximation error}
\label{sec:convergence}
In this paper we will not show the stability of the method, since this is already done in \cite{HohageNannen:09,Halla:16}. We merely focus on the best approximation error of this Galerkin method.
In \cite[Chapter~7.3]{ShenTangWang} it is shown that the error of interpolation by Laguerre functions decays super algebraically in the order of the Laguerre functions. Although this implies super algebraic convergence of our method,
it does not help us in choosing optimal parameters. Therefore, in this section we will derive estimates depending on the method parameters $\sigma$ and $R$ and the frequency $\omega$.

\subsection{Best approximation in one dimension}
\label{sec:BestAppr1d}
Before we discuss approximation results for the solutions in three dimensions we state some results 
regarding the simpler one dimensional problems.
\begin{thm}
  \label{thm:1d_approx}
  For $b\in\setC$, $\Re\of b>-1$ and $n\in\setN_0$,
  \begin{align*}
    \int_0^\infty\exp\of{-b\radvar}\phi_n(\radvar)\,d\radvar&=\frac{(b-1)^n}{(b+1)^{n+1}}.
    %\exp(-\sigma x)&=\sum_{n=0}^\infty 2\frac{(\sigma-1)^n}{(\sigma+1)^{n+1}}\phi_n(2x)\\
    %\exp(-\sigma x)&=\exp(-x)\sum_{n=0}^\infty\left(\frac{\sigma-1}{\sigma+1}\right)^n L_{n}^{(-1)}(2x)\\
    %\exp(-\sigma x)x&=\exp(-x)\sum_{n=0}^\infty-\frac{2n(\sigma-1)^{n-1}}{(\sigma+1)^{n+1}}L_{n}^{(-1)}(2x)\\
    %\exp(-\sigma x)x^2&=\exp(-x)\sum_{n=0}^\infty\frac{4n(\sigma-1)^{n-2}}{(\sigma+1)^{n+2}}(n-\sigma)L_{n}^{(-1)}(2x)
  \end{align*}
  For $\Re\of b> 0$ and $\radvar\in\setR_{\geq 0}$, we have
  $$
  \exp\of{-b\radvar}=\frac{2}{b+1}\sum_{n=0}^\infty\left(\frac{b-1}{b+1}\right)^n\phi_n(\radvar).
  $$
  Moreover, the $L^2\of{\setR_{\geq 0}}$-orthogonal projection onto $\tilde \X_N$ of $\exp\of{-b\cdot}$ is given by
  $$
  \Pi_N \exp\of{-b\cdot}=\frac{2}{b+1}\sum_{n=0}^N\left(\frac{b-1}{b+1}\right)^n\phi_n\of\cdot.
  $$
\end{thm}
\begin{proof}
  It is easily shown by partial integration and induction over $j$, that for $j\in\setN,j\leq n+1$
  \begin{multline*}
    \int_0^\infty\exp\of{-b\radvar}\phi_n(\radvar)\,d\radvar=\\
    \frac{1}{b+1}\sum_{k=0}^{j-1}\left(\frac{2}{b+1}\right)^kL_{n}^{(k)}(0)+\left(\frac{2}{b+1}\right)^j\int_0^\infty\exp\of{-\radvar(b+1)}L_{n}^{(j)}\of {2\radvar}\,d\radvar.
  \end{multline*}
  For $j=n+1$ we obtain
  \begin{align*}
    \int_0^\infty\exp\of{-b\radvar}\phi_n(\radvar)\,d\radvar&=\frac{1}{b+1}\sum_{k=0}^n\left(\frac{2}{b+1}\right)^kL_{n}^{(k)}(0)\\
    &\stackrel{\ref{lem:prop_laguerre}(\ref{enum:laguerre_rec})}{=}\frac{1}{b+1}\sum_{k=0}^n\left(-\frac{2}{b+1}\right)^kL_{n-k,k}(0)\\
&=\frac{1}{b+1}\sum_{k=0}^n\left(-\frac{2}{b+1}\right)^k\twovector{n}{k}\\
&=\frac{1}{b+1}\left(1-\frac{2}{b+1}\right)^n\\
&=\frac{(b-1)^n}{(b+1)^{n+1}}.
  \end{align*}
  If $\Re\of b>0$ we have $\exp(-b\cdot)\in L_2\of{\setR_{\geq 0}}$.
  Since $\{\phi_n,n\in\setN_0\}$ is a complete orthogonal system of $L^2\of{\setR_{\geq 0}},$ we have
    \begin{align*}
    \exp(-b\radvar)=\sum_{n=0}^\infty\frac{\left(\phi_n,\exp(-b\cdot)\right)_{L_2\of{\setR_{\geq 0}}}}{\left(\phi_n,\phi_n\right)_{L_2\of{\setR_{\geq 0}}}}\phi_n\of x=\frac{2}{b+1}\sum_{n=0}^\infty\left(\frac{b-1}{b+1}\right)^n\phi_n\of \radvar.
  \end{align*}
\end{proof}
\begin{cor}
  \label{cor:1d_approx}
  For $b\in\setC$, $\Re\of b>0$ and $N\in\setN_0$
    $$\inf_{u_N\in\tilde\X_N}\left\|\exp(-b\cdot)-u_N\right\|_{L_2\of{\setR_{\geq 0}}}\leq\|(I-\Pi_N)\exp\of{-b\cdot}\|_{L_2\of{\setR_{\geq 0}}}=\frac{1}{\sqrt{2\Re\of b}}\left|\frac{b-1}{b+1}\right|^{N+1}.$$
\end{cor}
\begin{proof}
  \begin{align*}
    \|(I-\Pi_N)\exp\of{-b \cdot}\|^2_{L^2\of{\setR_{\geq 0}}}&=\left\|\frac{2}{b+1}\sum_{n=N+1}^\infty\left(\frac{b-1}{b+1}\right)^n\phi_n\right\|^2_{L^2\of{\setR_{\geq 0}}}\\
    &=\left|\frac{2}{b+1}\right|^2\sum_{n=N+1}^\infty\left|\frac{b-1}{b+1}\right|^{2n}\left\|\phi_n\right\|^2_{L^2\of{\setR_{\geq 0}}}\\
    &=2\left|\frac{(b-1)^{N+1}}{(b+1)^{N+2}}\right|^2\sum_{n=0}^\infty\left|\frac{b-1}{b+1}\right|^{2n}\\
    &=2\left|\frac{(b-1)^{N+1}}{(b+1)^{N+2}}\right|^2\frac{1}{1-\left|\frac{b-1}{b+1}\right|^2}\\
    &=\left|\frac{b-1}{b+1}\right|^{2N+2}\frac{1}{2\Re\of b}.
  \end{align*}
\end{proof}
\begin{rmk}
  \label{rmk:1d_conv}
  Because of the representation of the solutions in the exterior
  $$u_\mathrm{ext}\of\radvar=\exp\of{\pm i\omega R}\exp\of{i\omega\sigma\radvar},$$
  for $d=1$, Theorem \ref{thm:1d_approx} and Corollary \ref{cor:1d_approx} (with $b=-i\sigma\omega$) state that the approximation by Laguerre functions in the $L^2$-norm depends on the quantity $\left|\frac{1+i\omega\sigma}{1-i\omega\sigma}\right|$. It is exact if $\omega\sigma=i$.
  In particular we have for $\Im\of{\sigma\omega}>0$
  \begin{align*}
    \inf_{u_h\in\tilde\X_N}\|\exp\of{i\sigma\omega \cdot}-u_h\|_{L^2\of{\setR_{\geq 0}}}
    &\leq\frac{1}{\sqrt{2\Im\of{\sigma\omega}}}\left|\frac{1+i\sigma\omega}{1-i\sigma\omega}\right|^{N+1},
  \end{align*}
  stating, that the best approximation error of outgoing solutions of the one dimensional Helmholtz equation decreases exponentially with respect to the number of exterior degrees of freedom $N$. 
  
  A similar approximation result for Hardy space infinite elements can be found in \cite{HohageNannen:09}. Note, that a different variational framework is used there. Nevertheless, as explained in Subsection \ref{sec:hsm}, this framework is to some extent equivalent to the one used here.
\end{rmk}

\subsection{Best approximation of the zeroth spherical Hankel function}
Since the radial part of the solutions of the three dimensional Helmholtz resonance problem in the exterior domain consist of spherical Hankel functions of the first kind, we proceed by discussing the approximation of $h_0^{(1)}$ by Laguerre functions.
\begin{defn}
  \label{defn:hankel}
  For $n\in\setN_0$, the spherical Hankel functions of the first kind $h^{(1)}_\nu$ can be defined by
  $$
  h_\nu(\radvar):=-\frac{i}{\radvar}\exp(i\radvar)\tilde h_\nu(\radvar),
  $$
  with
    $$
  \tilde h_\nu(\radvar):=(-i)^{\nu}\sum_{m=0}^n\frac{i^m}{m!(2\radvar)^m}\frac{(\nu+m)!}{(\nu-m)!)},
  $$
  (cf. \cite[Section 2.4]{ColtonKress:98}).
  To simplify the notation we will omit the superscript $\,^{(1)}$ and simply write $h_\nu:=h^{(1)}_\nu$.
\end{defn}
Suppose we want to approximate 
$$h_0(\omega R+i\radvar)=\frac{-\exp\of{i\omega R}\exp\of{-\radvar}}{-i\omega R+\radvar},$$
using our basis functions $\phi_n$.
This would be the case if we applied a frequency dependent complex scaling $\sigma\of\omega=\frac{\iu}{\omega}$ (cf. \cite{nw2018}).
Then the approximation error will be governed by the terms $\left(\frac{\exp\of{-\cdot}}{a+\cdot},\phi_n\right)_{L_2\of{\setR_{\geq 0}}}$, for $a=-i\omega R$.
This motivates the following definition.
\begin{defn}
  For $a\in\setC\setminus\setR_{\leq 0}$ and $n,k\in\setN_0$, we define
  \begin{align*}
    \alpha_{n,k}(a)&:=\int_0^\infty \frac{\exp(-\radvar)}{(a+\radvar)^k}\phi_n(\radvar)\,d\radvar,
    %\alpha_n(2a)&:=\beta_{n,0}(2a)
  \end{align*}
\end{defn}
%\begin{defn}
%  For $a,b,c\in\setC$ with $\Re\of a,\Re\of x>0$ the \emph{confluent hypergeometric function of the second kind} is usually defined as
%  \begin{align}
%    \mathrm{U}(a,b,x):=\frac{1}{\Gamma\of a}\int_0^\infty\frac{t^{a-1}\exp(-xt)}{(1+t)^{a+1-b}}\,dt.
%  \end{align}
%\end{defn}
%\begin{rmk}
%  For $n\in\setN,m\in\setZ$, the function $x\mapsto U\of{n+1,m,x}$ is analytic on $\setC\setminus\setR_{<0}$ and given by
%  $$U(n+1,m,x)=\frac{x^{1-m}}{n!}\int_0^\infty\frac{t^n\exp\of{-t}}{(x+t)^{n+2-m}}\,dt.$$
%\end{rmk}
The following lemma shows, that the numbers $\alpha_{n,1}\of a$ can be calculated by a simple integral.
\begin{lem}
  \label{lem:alpha_calc}
    For $a\in\setC\setminus\setR_{\leq 0}$ and $n\in\setN_0$, we have
    \begin{align*}
      \alpha_{n,1}(a)&=\int_0^\infty\frac{\radvar^n\exp(-\radvar)}{(2a+\radvar)^{n+1}}\,d\radvar%=n!\mathrm{U}(n+1,1,x).
    \end{align*}
    The numbers $\alpha_{n,1}\of a$ are the coefficients of the expansion of $\frac{\exp(-\cdot)}{a+\cdot}$ in the Laguerre functions $\phi_n$ and therefore
    $$
    \frac{\exp(-\radvar)}{a+\radvar}=2\sum_{n=0}^\infty\alpha_{n,1}(a)\phi_n(\radvar).
    $$
  \end{lem}
\begin{proof}
  It is easily shown by partial integration and induction in $j$, that for $j\leq n$
  \begin{align*}
    \int_{0}^\infty\frac{t^n\exp\of{-t}}{(2a+t)^{n+1}}\,dt &= \frac{(n-j)!}{n!}\int_0^\infty\frac{1}{(2a+t)^{n+1-j}}\frac{d^j}{dt^j}\left(\exp\of{-t}t^n\right)\,dt.
  \end{align*}
  For $j=n$ we obtain
  \begin{align*}
    \int_{0}^\infty\frac{t^n\exp\of{-t}}{(2a+t)^{n+1}}\,dt &= \frac{1}{n!}\int_0^\infty\frac{1}{2a+t}\frac{d^n}{dt^n}\left(\exp\of{-t}t^n\right)\,dt\\
    &\stackrel{\ref{lem:prop_laguerre}(\ref{enum:laguerre_rodriguez})}{=}\int_0^\infty\frac{1}{2a+t}\exp\of{-t}L_n\of{t}\,dt\\
    &=\int_0^\infty\frac{2\exp\of{-t}}{2a+2t}\phi_n\of{t}\,dt=\alpha_{n,1}\of a.
  \end{align*}
\end{proof}
%\begin{lem}
%  The functions $x\mapsto\alpha_{n,k}(x)$ have the explicit representation
%    \begin{align}
%      \alpha_{n,k}(x)&=2^{k-1}\twovector{k+n-1}{n}\int_0^\infty\frac{t^n\exp(-t)}{(x+t)^{n+k}}\,dt\\
%      &=(4x)^{k-1}\frac{(k+n-1)!}{(k-1)!}\mathrm{U}(n+1,2-k,2x).
%    \end{align}
%   They are the coefficients of the expansion of $\frac{\exp(-x)}{a+x}$ in the Laguerre functions $\phi_n$ and therefore
%    $$
%    \frac{\exp(-x)}{(a+x)^k}=2\sum_{n=0}^\infty\alpha_{n,k}(a)\phi_n(x).
%    $$
%\end{lem}
%\begin{proof}
%  It is easily shown by induction, that for $j\leq n$
%  \begin{align}
%    \int_0^\infty\frac{1}{(2x+t)^{k}}\frac{d^n}{dt^n}\left(\exp\of{-t}t^n\right)\,dt&=
%    \frac{(k+j-1)!}{(k-1)!}\int_0^\infty\frac{1}{(2x+t)^{k+j}}\frac{d^{n-j}}{dt^{n-j}}\left(\exp\of{-t}t^n\right)\,dt
% \end{align}
%  For $j=n$ we obtain
%  \begin{align}
%    \frac{(k+n-1)!}{(k-1)!}\int_0^\infty\frac{\exp\of{-t}t^n}{(2x+t)^{k+n}}\,dt&=\int_0^\infty\frac{1}{(2x+t)^{k}}\frac{d^n}{dt^n}\left(\exp\of{-t}t^n\right)\,dt\\
%    &=n!\int_0^\infty\frac{\exp(-t)}{(2x+t)^k}L_n(t)\,dt\\
%    &=\frac{n!}{2^{k-1}}\int_0^\infty\frac{\exp(-2t)}{(x+t)^k}L_n(2t)\,dt\\
%    &=\frac{n!}{2^{k-1}}\alpha_{n,k}(x)
%  \end{align}
%\end{proof}
The following theorem gives an asymptotic expansion of the terms $\alpha_{n,1}\of a$ with respect to $n$.
\begin{thm}[asymptotic behavior of $\alpha_{n,1}$]
  \label{thm:alpha_asympt}
  For $a\in\setC\setminus\setR_{\leq 0}$
  \begin{align*}
  \alpha_{n,1}(a)&=\exp\left(a-2\sqrt{2a(n+1)}\right)\frac{\sqrt \pi}{(2a(n+1))^\frac{1}{4}}\left(1+\mathcal O\of{\frac{1}{\sqrt{n+1}}}\right),&n&\to\infty.
\end{align*}
The symbols $\sqrt{z}$ and $z^\frac{1}{4}$ for $z\in\setC\setminus\setR_{\leq 0}$ assume their respective principal values (their image is symmetric with respect to the positive real axis).
\end{thm}
\begin{proof}
  Lemma \ref{lem:alpha_calc} states that 
  $$\alpha_{n,1}\of a=n!\mathrm{U}\of{n+1,1,2a},$$
  where for  $n\in\setN_0$, $a\in\setC\setminus\setR_{\leq 0}$
  $$U(n+1,m,a)=\frac{a^{1-m}}{n!}\int_0^\infty\frac{t^n\exp\of{-t}}{(a+t)^{n+2-m}}\,dt.$$
  The function $\mathrm U$ is called \emph{confluent hypergeometric function of the second kind}.
  Using (10.3.39) and (9.1.3) in \cite{Temme} we obtain 

  \begin{align*}
    n!U(n+1,1,2a)&=2\exp\of{a}\left(K_{0}\of{2\sqrt{2a(n+1)}}+\mathcal O\of{\frac{1}{\sqrt{n+1}}}\right),
  \end{align*}
  for $n\to\infty$ and
  \begin{align*}
    K_{n}\of{z}=\sqrt{\frac{\pi}{2z}}\exp\of{-z}\left(1+\mathcal O\of{\frac{1}{z}}\right),
  \end{align*}
  for $|z|\to\infty$.
  All in all we obtain
  \begin{align*}
    \alpha_{n,1}(a)&=n!U(n+1,1,2a)\\
    &=\sqrt\pi(2a(n+1))^{-\frac{1}{4}}\exp\of{a-2\sqrt{2a(n+1)}}\left(1+\mathcal O\of{\frac{1}{\sqrt{n+1}}}\right),
  \end{align*}
  for $n\to\infty$.
\end{proof}
Using the lemma above we can now bound the best approximation error of $h_0\of{\omega R+i\cdot}$ by Laguerre functions.
\begin{lem}
  \label{lem:hankel0_m1}
  Let $R>0$, $N\in\setN$ and $\omega\in\setC$. Then there exists $c>0$ independent of $N$, such that
    $$\left\|(I-\Pi_N)h_0\of{\omega R+i\cdot}\right\|_{L_2\of{\setR_{\geq 0}}}\leq\frac{c\sqrt\pi}{(2|\omega|R)^{\frac{1}{4}}}\exp\of{-2\Re\sqrt{\omega R(N+1)}}.$$
  \end{lem}
\begin{proof}
  \begin{align*}
    \left\|(I-\Pi_N)h_0\of{\omega R+i\cdot}\right\|_{L_2\of{\setR_{\geq 0}}}^2&=
    \left\|(I-\Pi_N)(-1)\exp\of{i \omega R}\frac{\exp\of{-\cdot}}{-i\omega R+\cdot}\right\|_{L_2\of{\setR_{\geq 0}}}^2\\
    &=\sum_{n=N+1}^\infty|2\alpha_{n,1}\of{-i\omega R}|^2\|\phi_n\|^2\\
    &\leq 2c\sum_{n=N+1}^\infty\left|\exp\left(-i\omega R-2\sqrt{-2i\omega R(n+1)}\right)\right.\\
    {}&\times\left.\frac{\sqrt \pi}{(-2i\omega R(n+1))^\frac{1}{4}}\right|^2\\
    &=c\pi\sqrt{\frac{2}{|\omega| R}}\sum_{n=N+1}^\infty\frac{\exp\left(-4\Re\sqrt{\omega R(n+1)}\right)}{\sqrt{n+1}}.
  \end{align*}
  Since the summand is a decreasing function in $n$ we can replace the sum with an integral and obtain
  \begin{align*}
    \left\|(I-\Pi_N)h_0\of{\omega R+i\cdot}\right\|_{L_2\of{\setR_{\geq 0}}}^2&\leq
    c\pi\sqrt{\frac{2}{\omega R}}\int_{N+1}^\infty\frac{\exp\left(-4\sqrt{\omega Rt}\right)}{\sqrt{t}}dt\\
    &=2c\pi\sqrt{\frac{2}{\omega R}}\int_{\sqrt {N+1}}^\infty\exp\of{-4\sqrt {\omega R} s}ds\\
    &=\frac{c\pi}{\sqrt {2\omega R}}\exp\of{-4\sqrt{\omega R(N+1)}}.
  \end{align*}
\end{proof}
Lemma \ref{lem:hankel0_m1} only gives us the approximation error for a frequency dependent complex scaling $\sigma\of{\omega}=\frac{i}{\omega}$ which leads to solutions with exponential decay $\exp(-\cdot)$. General scalings of the form (\ref{eq:scaling}) result in solutions with exponential decay $\exp\of{\Re\of{i\omega\sigma}}$ leading to an additional error term.
\begin{thm}
  \label{thm:hankelz_split}
  Let $\omega,\sigma\in\setC$ with $\Im\of{\omega\sigma}>0$, and $R>0$ and $N\in\setN$. Then there exist constants $C_1,C_2>0$ independent of $N$, such that the best approximation error of 
  $h_0\of{\omega R+\omega\sigma\cdot}$ can be bounded by
  \begin{align*}
    \inf_{u_{N}\in\tilde\X_N}\|h_0\of{\omega R+i\omega\sigma\cdot}-u_N\|_{L_2\of{\setR_{\geq 0}}}\leq C_1\left|\frac{1+i\sigma\omega}{1-i\sigma\omega}\right|^{N+1}+C_2\varepsilon\of{N,\frac{R}{\sigma},-i\omega\sigma}
  \end{align*}
  with
  \begin{align*}
    \varepsilon\of{N,a,b}:=\|(I-\Pi_N)\frac{1}{a+\cdot}\Pi_N\exp\of{-b\cdot}\|_{L_2\of{\setR_{\geq 0}}}.
  \end{align*}
\end{thm}
\begin{proof}
  \begin{align*}
    \|(I-\Pi_N)h_0\of{\omega R+\omega\sigma\cdot}\|_{L_2\of{\setR_{\geq 0}}}=&
  \|(I-\Pi_N)\frac{-i\exp\of{i\omega R}}{\omega\sigma}\frac{\exp\of{i\omega\sigma\cdot}}{\frac{R}{\sigma}+\cdot}\|_{L_2\of{\setR_{\geq 0}}}\\
  \leq& C_2\|(I-\Pi_N)\frac{1}{\frac{R}{\sigma}+\cdot}(I-\Pi_N){\exp\of{i\omega\sigma\cdot}}\|_{L_2\of{\setR_{\geq 0}}}\\
  &{}+C_2\underbrace{\|(I-\Pi_N)\frac{1}{\frac{R}{\sigma}+\cdot}\Pi_N{\exp\of{i\omega\sigma\cdot}}\|_{L_2\of{\setR_{\geq 0}}}}_{\varepsilon\of{N,\frac{R}{\sigma},-i\omega\sigma}},
\end{align*}
with $C_2=\left|\frac{\exp\of{i\omega R}}{\omega\sigma}\right|$.
The first term can be bounded by
\begin{align*}
  \|(I-\Pi_N)\frac{1}{\frac{R}{\sigma}+\cdot}(I-\Pi_N){\exp\of{i\omega\sigma\cdot}}\|_{L_2\of{\setR_{\geq 0}}}&\leq 2C\|(I-\Pi_N){\exp\of{i\omega\sigma\cdot}}\|_{L_2\of{\setR_{\geq 0}}}\\
  &\leq2C\frac{1}{\sqrt{2\Im\of{\omega\sigma}}}\left|\frac{1+i\omega\sigma}{1-i\omega\sigma}\right|^{N+1},
\end{align*}
for some constant $C>0$ (cf. Corollary \ref{cor:1d_approx}).
\end{proof}

To obtain a bound for the best approximation error of $h_0$ in the space $\tilde X_N$ we need to find a bound for the expression $\varepsilon$. Since for $b\in\setC$, $a\in\setC\setminus\setR_{\leq 0}$
\begin{align*}
  (I-\Pi_N)\frac{1}{a+\cdot}\Pi_N\exp\of{-b\cdot}=\sum_{n=N+1}^\infty\int_0^\infty\frac{1}{a+\radvar}\sum_{k=0}^N\frac{(1-b)^{k}}{(1+b)^{k+1}}\phi_k\of\radvar\phi_n\of\radvar\,d\radvar\phi_n\of\cdot,
\end{align*}
we need to know the asymptotic behavior of the expressions
\begin{align}
  \label{eq:beta}
    \beta_{n,k}(a):=\int_0^\infty\frac{\phi_n(x)\phi_k(x)}{a+x}\,dx,
\end{align}
for large $n\in\setN$.
Thus, we state the following lemma.

\begin{lem}
  \label{lem:beta}
  Let $a\in\setC\setminus\setR_{<0}$, and $\beta_{n,k}$ given by (\ref{eq:beta}).
  Then for $n\geq k$ there holds
  \begin{align}
    \beta_{n,k}\of a&=\alpha_{n,1}(a)L_k(-2a).
  \end{align}
\end{lem}
\begin{proof}
  We prove by induction in $k$. For $k=0$ we have
  \begin{align*}
    \beta_{n,0}(a)=\alpha_{n,1}(a)=\alpha_n(a)L_0(-2a),
  \end{align*}
  and for $k=1$, $n\geq 1$
  \begin{align*}
    \beta_{n,1}(a)&=\int_0^\infty\frac{\exp(-x)(1-2x)}{a+x}\phi_n(x)\,dx\\
    &=\alpha_{n,1}(a)-2\int_0^\infty\frac{x+a-a}{a+x}\exp(-x)\phi_n(x)\,dx\\
    &=\alpha_{n,1}(a)-\delta_{n,0}+2a\alpha_{n,1}(a)\\
    &=(2a+1)\alpha_{n,1}(a)=L_1(-2a)\alpha_{n,1}(a).
  \end{align*}
  For $n\geq k$ we use the recursion (\ref{eq:laguerre_rec}) and write
  \begin{align*}
    \beta_{n,k}(a)&=\int_0^\infty\frac{\phi_n(x)}{a+x}\left(\frac{2k-1-2x}{k}\phi_{k-1}(x)-\frac{k-1}{k}\phi_{k-2}(x)\right)\,dx=\\
    &=\frac{2k-1}{k}\beta_{n,k-1}(a)-\frac{k-1}{k}\beta_{n,k-2}(a)-\frac{2}{k}\int_0^\infty\frac{x+a-a}{a+x}\phi_n(x)\phi_{k-1}(x)\,dx\\
    &=\frac{2k-1}{k}\beta_{n,k-1}(a)-\frac{k-1}{k}\beta_{n,k-2}(a)+\frac{2a}{k}\beta_{n,k-1}(a)\\
    &=\alpha_{n,1}(a)\left(\frac{2k-1+2a}{k}L_{k-1}(-2a)-\frac{k-1}{k}L_{k-2}(-2a)\right)\\
    &=\alpha_{n,1}(a)L_k(-2a).
  \end{align*}
\end{proof}
Now we are able to bound the term $\varepsilon\of{N,a,b}$.
\begin{lem}
  \label{lem:diverror}
  For $a\in\setC\setminus\setR_{\leq 0}$, $b\in\setC, \Re\of b>0$ , there exists $C>0$ independent of $N$, such that
  $$\varepsilon\of{N,a,b}\leq C\exp\of{-2\Re\sqrt{2a(N+1)}}.$$
\end{lem}
\begin{proof}
  Since
  \begin{align*}
    \Pi_N\exp\of{-bx}=\frac{2}{b+1}\sum_{k=0}^N\left(\frac{b-1}{b+1}\right)^k\phi_k(x),
  \end{align*}
  we have
  \begin{align*}
    \varepsilon\of{N,a,b}^2&=\left\|\sum_{n=N+1}^\infty \frac{4}{b+1}\sum_{k=0}^N\left(\frac{b-1}{b+1}\right)^k\left(\phi_n,\frac{\phi_k}{\cdot+a}\right)_{L^2\of{\setR_{\geq 0}}}\phi_n\right\|^2_{L^2\of{\setR_{\geq 0}}}\\
    &=\left\|\sum_{n=N+1}^\infty\frac{4}{b+1}\sum_{k=0}^N\left(\frac{b-1}{b+1}\right)^k \beta_{n,k}(a)\phi_n\right\|^2_{L^2\of{\setR_{\geq 0}}}\\
    &=\frac{1}{2}\sum_{n=N+1}^\infty\left|\frac{4}{b+1}\alpha_{n,1}(a)\sum_{k=0}^N\left(\frac{b-1}{b+1}\right)^kL_k(-2a)\right|^2.
  \end{align*}
  Using the generating function of the Laguerre polynomials from Lemma \ref{lem:prop_laguerre} (\ref{enum:laguerre_gen})
  we have for some constant $C$
  \begin{align*}
    \sum_{k=0}^N\left(\frac{b-1}{b+1}L_k\of{-2a}\right)^k&\leq C\frac{1}{1-\frac{b-1}{b+1}}\exp\of{2a\frac{\frac{b-1}{b+1}}{1-\frac{b-1}{b+1}}}\\
    &=C\frac{b+1}{2}\exp\of{a(b-1)},
  \end{align*}
  and thus
  \begin{align*}
    \varepsilon\of{N,a,b}^2&\leq2C\sum_{n=N+1}^\infty\left|\alpha_{n,1}(a)\exp\of{a(b-1)}\right|^2.
  \end{align*}
  Substituting the asymptotic behavior of $\alpha_{n,1}$ and repeating the arguments of the proof of Lemma \ref{lem:hankel0_m1} we find for some constant $c\in\setR$
  \begin{align*}
    \sum_{n=N+1}^\infty\left|\alpha_{n,1}\of{a}\right|^2
    \leq c\left|\exp\of{2a}\frac{\pi}{\sqrt{2a}}\right|\frac{1}{2\Re\of{\sqrt{2a}}}\exp\of{-4\Re\of{\sqrt{2a(N+1)}}}.
  \end{align*}
  All in all this gives
  \begin{align*}
    \varepsilon\of{N,a,b}\leq \tilde C \frac{\exp\of{\Re\of{ab}}\sqrt\pi}{\sqrt{|2a|\Re\of{\sqrt{2a}}}}\exp\of{-2\Re\of{\sqrt{2a(N+1)}}}.
  \end{align*}
\end{proof}

Using Theorem \ref{thm:hankelz_split} and Lemma \ref{lem:diverror} we have proven 
\begin{thm}
  \label{thm:hankel0_final_error}
  For  $R>0$, and $\omega,\sigma\in\setC$, s.t. $\Im\of{\sigma\omega}>0$, we can bound the approximation error of the complex scaled zeroth spherical Hankel function by
  \begin{align*}
    \inf_{u_N\in \radmod X_N}\left\|h_0\of{\omega(R+\sigma x)}-u_N\right\|_{\setR_{\geq 0}}\leq c_1\left|\frac{1+i\sigma\omega}{1-i\sigma\omega}\right|^{N+1}+c_2\exp\of{-2\Re\sqrt{\frac{2R(N+1)}{\sigma}}},
  \end{align*}
  for some constants $c_1,c_2>0$ independent of $N$.
\end{thm}
Theorem \ref{thm:hankel0_final_error} shows, that the approximation error of $h_0\of{\omega R+\sigma\omega\cdot}$ by Laguerre functions can be split up into two parts: 
  \begin{enumerate}[(i)]
    \item{An exponentially decaying part similar to the error of the one dimensional problem, which is generated by the different exponential decay of the solution and the basis functions and}
    \item{a super algebraic part due to the fact that we approximate the rational part of $h_0$ by polynomials.}
\end{enumerate}
        
\subsection{Approximation of spherical Hankel functions with higher index}
Up to now we have only dealt with the approximation of the complex scaled Hankel function with index zero. For Hankel functions with higher indices similar bounds for the approximation error can be derived.

\begin{thm}
  \label{thm:hankel_approx_higher}
  The approximation error of the complex scaled spherical Hankel functions can be bounded by
    $$
  \inf_{u_N\in \radmod\X_N}\left\|h_\nu\of{\omega(R+\sigma\cdot)}-u_N\right\|_{\setR_{\geq 0}}\leq c_1\left|\frac{1+i\sigma\omega}{1-i\sigma\omega}\right|^{N+1}+c_2\exp\of{-2\Re\sqrt{\frac{2R(N+1)}{\sigma}}}(N+1)^\frac{\nu}{2},
  $$
  for some constants $c_1,c_2>0$ independent of $N$.
\end{thm}
\begin{proof}
  We skip the technical details and give just a short sketch of the proof. Similar to Lemma \ref{lem:beta} one can show that
    $$\int_0^\infty \frac{1}{(a+\radvar)^j}\phi_n(\radvar)\phi_k(\radvar)\,d\radvar=\alpha_{n,1}(a)L_k\of{-2a}-\frac{2}{k}\alpha_{n,j-1}(a).$$
    Moreover similar to Theorem \ref{thm:alpha_asympt} one can show an asymptotic behavior of 
     $\alpha_{n,k}$ of the form
  $$|\alpha_{n,k}(a)|\leq C_{k,a}\left|\exp\of{-2\sqrt{2a(n+1)}}\right|(n+1)^{-\frac{3}{4}+\frac{k}{2}},\quad n\to\infty.$$
  Using this and similar ideas as in the proofs of Theorem \ref{thm:hankelz_split} and Lemma \ref{lem:diverror} leads to the desired result. 
\end{proof}

  \subsection{General approximation results}
  In the previous sections we have derived approximation results for the spherical Hankel functions with fixed index. Since due to (\ref{eq:radiation_cond}) our solutions might consist of infinite sums over all spherical Hankel functions for a complete analysis results dealing with the uniformity of the coefficients $\beta_{\nu,j}$ from (\ref{eq:radiation_cond}) are missing.

\section{Numerical Experiments}
\label{sec:numerics}

In the following we illustrate our theoretical findings from the previous sections by numerical examples. All numerical examples were computed using the high order finite element software NGSolve \cite{ngsolve:14} and the mesh generator Netgen \cite{netgen}.

  \subsection{Approximation of known solutions of the Helmholtz equation}
  First, we compare our theoretical results with numerical ones in situations, where exact solutions are known. The aim of this subsection is to highlight the dependency of the error on the different parameters. We start with the approximation of the Hankel function $h_0$.

Please note, that some of the following figures exhibit horizontal axes with units that scale like square roots. This helps visualizing the super algebraic decay of some of the error terms.

Figure \ref{fig:hankel_approx} shows the approximation error of the complex scaled zeroth Hankel function $h_0$
$$\text{error}\of N=\left\|(I-\Pi_N)h_0(\omega R+\omega\sigma\cdot)\right\|_{L_2\of{\setR_{\geq 0}}}.$$
The error was computed using numerical integration.
The results coincide nicely with the theoretical results from Theorem \ref{thm:hankel0_final_error}. 
In Figure \ref{fig:hankel_approx_exp} the parameters are chosen such that the error is dominated by the exponential term. We can observe that the exponential decay depends only on the parameters $\sigma$ and $\omega$ and is independent of $R$. Figure \ref{fig:hankel_approx_sqrt} shows approximation errors in a regime where the super algebraic part of the error dominates. The decay here is independent of the frequency $\omega$.

\begin{figure}[H]
  \begin{subfigure}{\textwidth}
    \centering
  \includegraphics{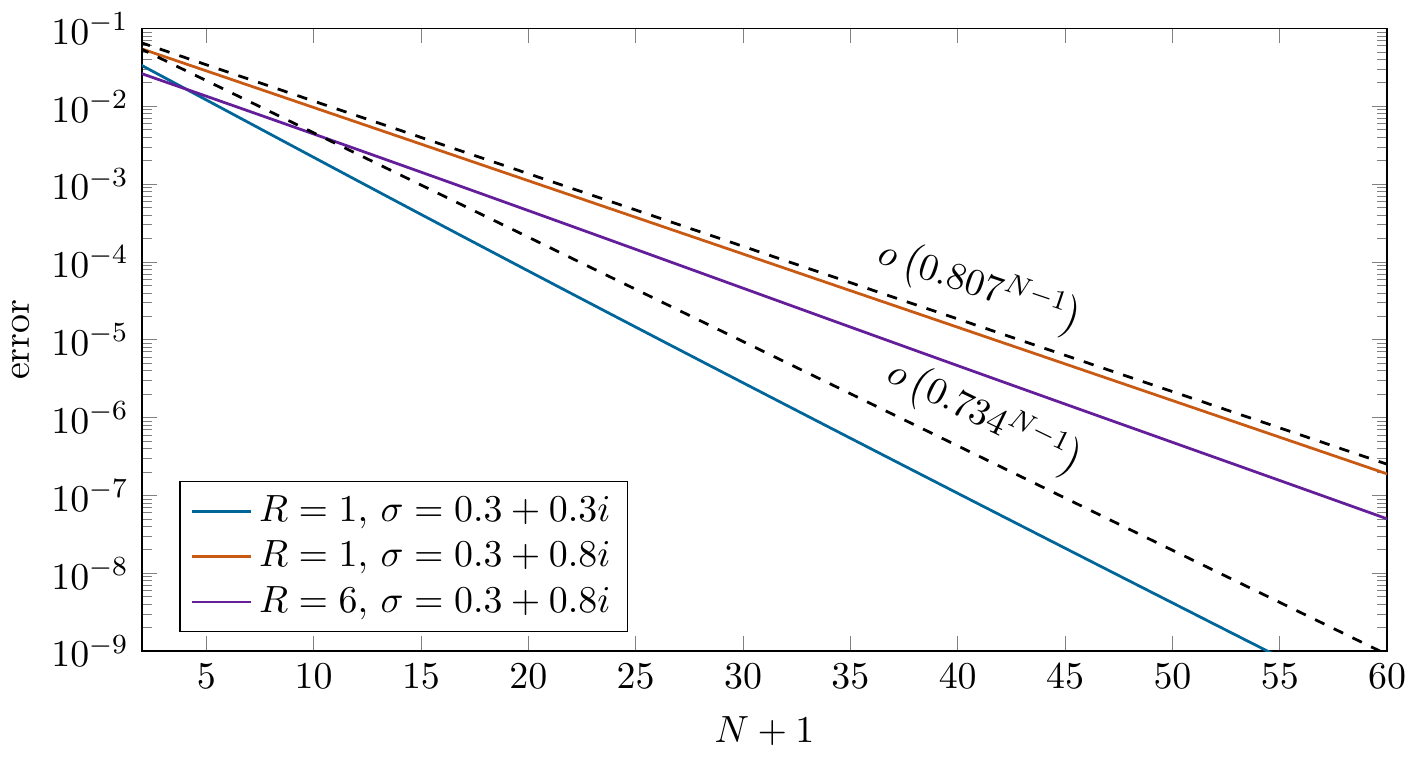}
  \caption{Exponential convergence for $\omega=10-0.5i$ and varying $R$ and $\sigma$. }
  \label{fig:hankel_approx_exp}
  \end{subfigure}
  \begin{subfigure}{\textwidth}
    \centering
    \includegraphics{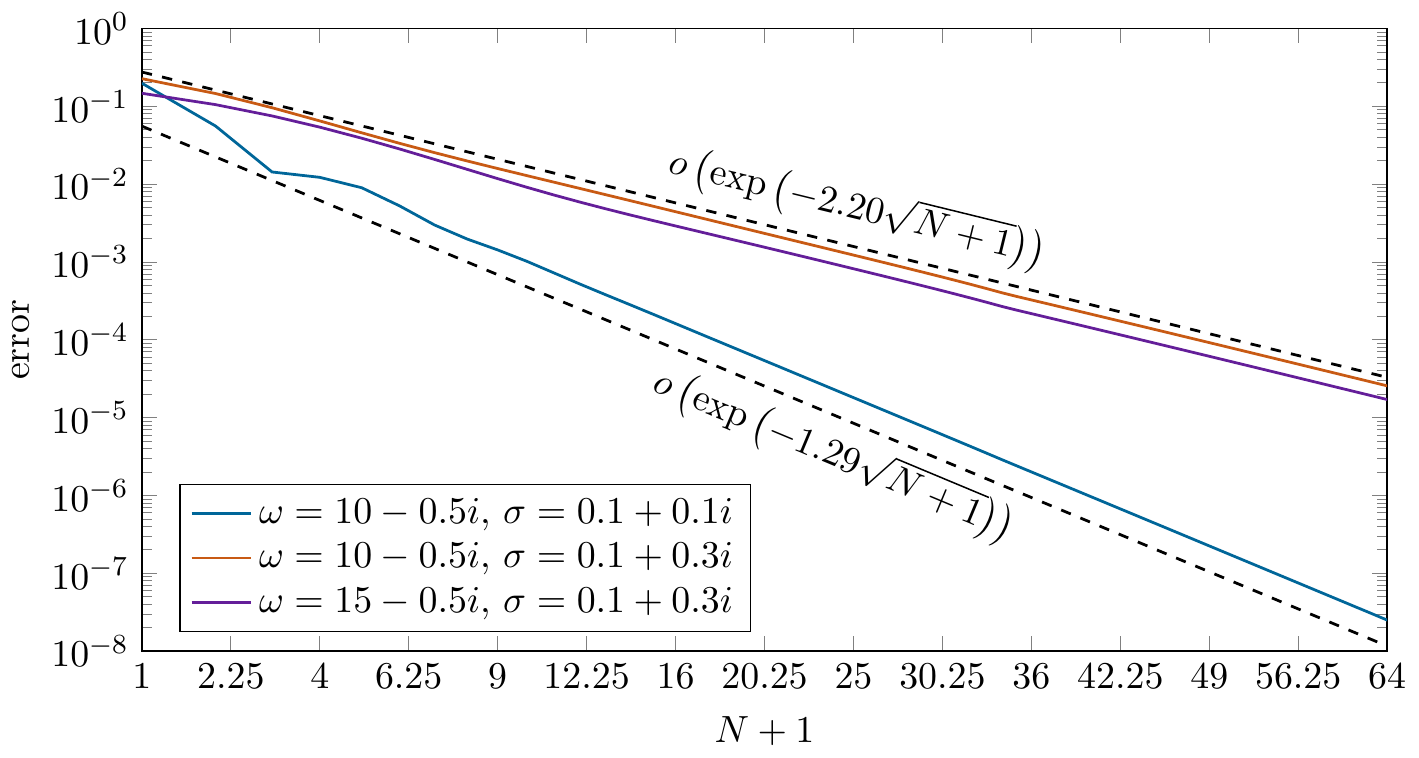}
    \caption{Super algebraic convergence for $R=0.1$ and varying parameters $\omega$ and $\sigma$. Note that the units on the horizontal axis scale like square roots.}
    \label{fig:hankel_approx_sqrt}
  \end{subfigure}
  \caption{Approximation error of $h_0(\omega(R+\sigma\cdot))$ for varying parameters. The dashed lines mark the predicted convergence rates from Section \ref{sec:convergence}.}
  \label{fig:hankel_approx}
\end{figure}
Figure \ref{fig:hankel_inds} shows the approximation of Hankel functions with different indices and exhibits the predicted behavior from Theorem \ref{thm:hankel_approx_higher}. Again we chose parameters such that the exponential error and the super algebraic error dominates in Figure \ref{fig:hankel_inds_exp} and Figure \ref{fig:hankel_inds_sqrt} respectively. The experiments also show that the exponential decay is independent of the index of the Hankel function $\nu$.

Figure \ref{fig:condition_numbers} shows the condition numbers of the discretization matrices of the bilinear forms
$$\tilde s^\sigma+\lambda_\nu\tilde m_0^\sigma-\omega^2\tilde m_1^\sigma,$$
with respect to different infinite element orders $N$. These matrices correspond to discretizations of the spherical Bessel equations with index $\lambda_\nu=\nu(\nu+1)$ (cf. Section \ref{sec:example}). The condition numbers grow slower than $o\of{N^3}$. For $60$ degrees of freedom the condition number is about $10^{5}$. But since the best approximation error decays at least super algebraically, the mild grow in the condition number is dominated by the fast convergence of the approximation error (cf. Figure \ref{fig:hankel_approx}).
\begin{figure}[H]
  \begin{subfigure}{\textwidth}
    \centering
    \includegraphics[width=0.8\textwidth]{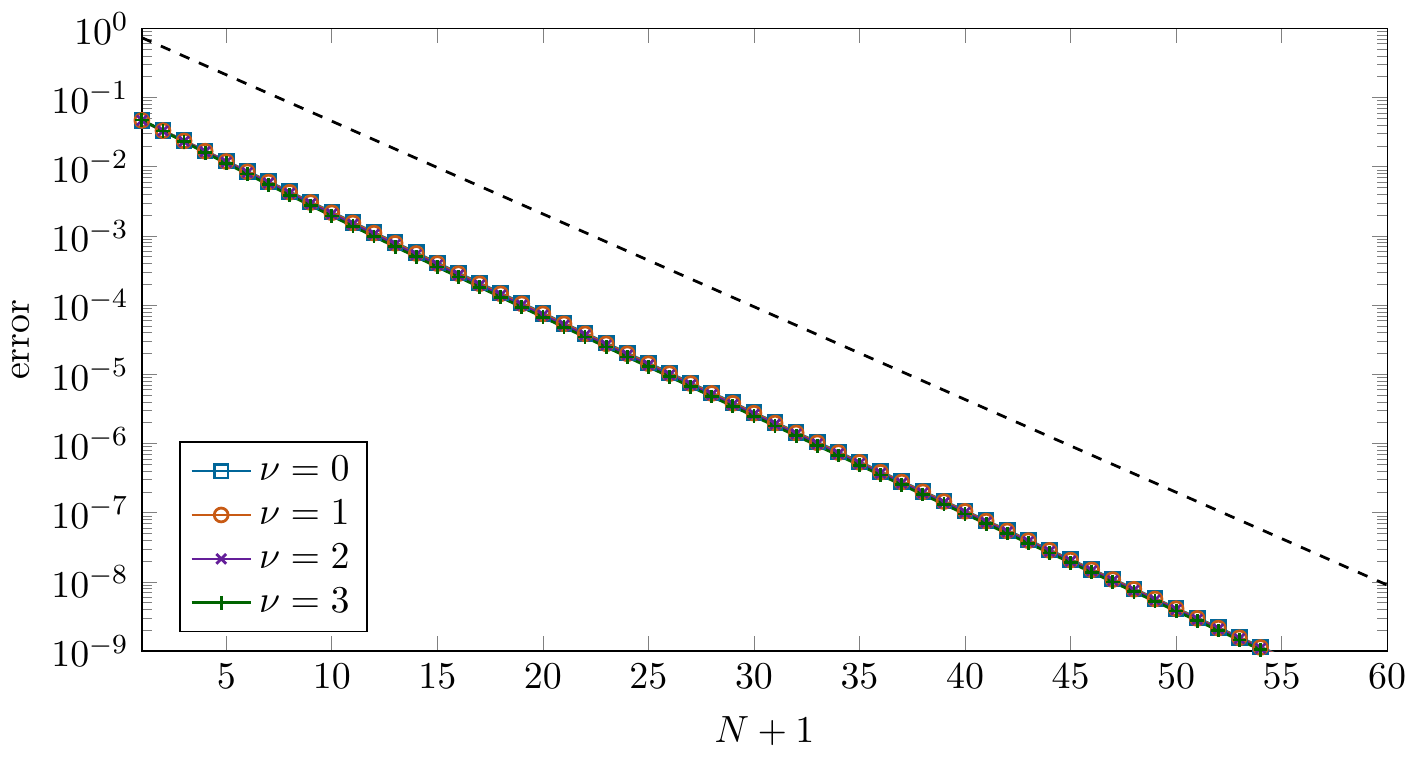}
    \caption{Exponential convergence for $\omega=10-0.5i$, $\sigma=0.3+0.3i$ and $R=1$.}
    \label{fig:hankel_inds_exp}
  \end{subfigure}
  \begin{subfigure}{\textwidth}
    \centering
    \includegraphics[width=0.8\textwidth]{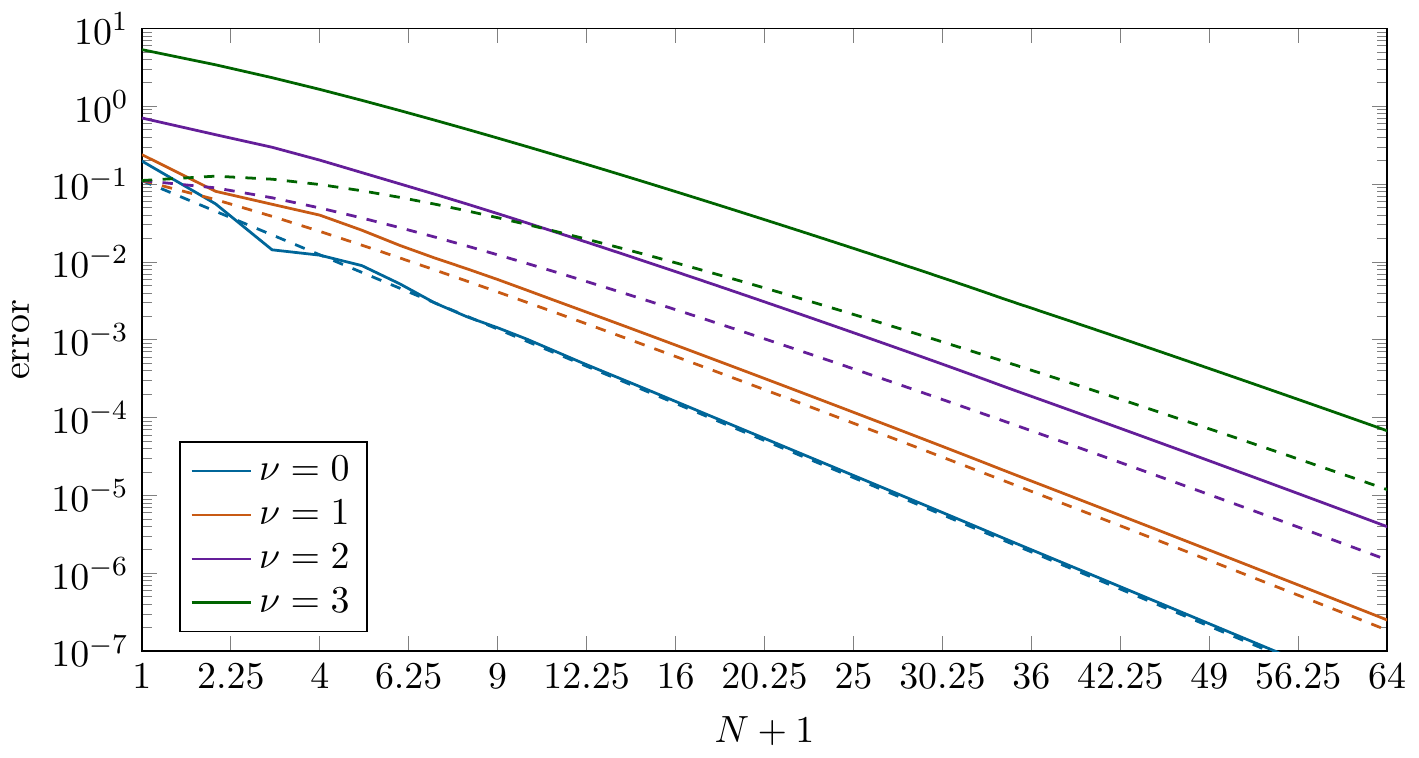}
    \caption{Super algebraic convergence for $\omega=10-0.5i$, $\sigma=0.1+0.1i$ and $R=0.1$. Note that the units on the horizontal axis scale like square roots.}
    \label{fig:hankel_inds_sqrt}
  \end{subfigure}
  \caption{Approximation error of $h_\nu(\omega(R+i\sigma\cdot))$. The dashed lines mark the predicted convergence rates from Section \ref{sec:convergence}.}
  \label{fig:hankel_inds}
\end{figure}

Figure \ref{fig:convergence_sqrt} shows the convergence in the number of unknowns for one selected eigenvalue of the separated problem (cf. Subsection \ref{sec:example}). Again we chose parameters such that the super algebraic part of the approximation error dominates. The expected rate of convergence for the eigenvalues is the approximation error squared (\cite{BabuskaOsborn}). The results in Figure \ref{fig:convergence} show exactly this behavior. 

In Figure \ref{fig:convergence_r} we added an interior domain to the problem. Since the super algebraic part of the error declines faster, the exponential error dominates here. The faster convergence in the plotted region comes with the price of more unknowns due to the additional interior domain.

%Again the convergence agrees with the predicted error from the approximation results (the approximation error squared cf. \cite{ColMonk}).

\begin{figure}[H]
    \centering
  \includegraphics{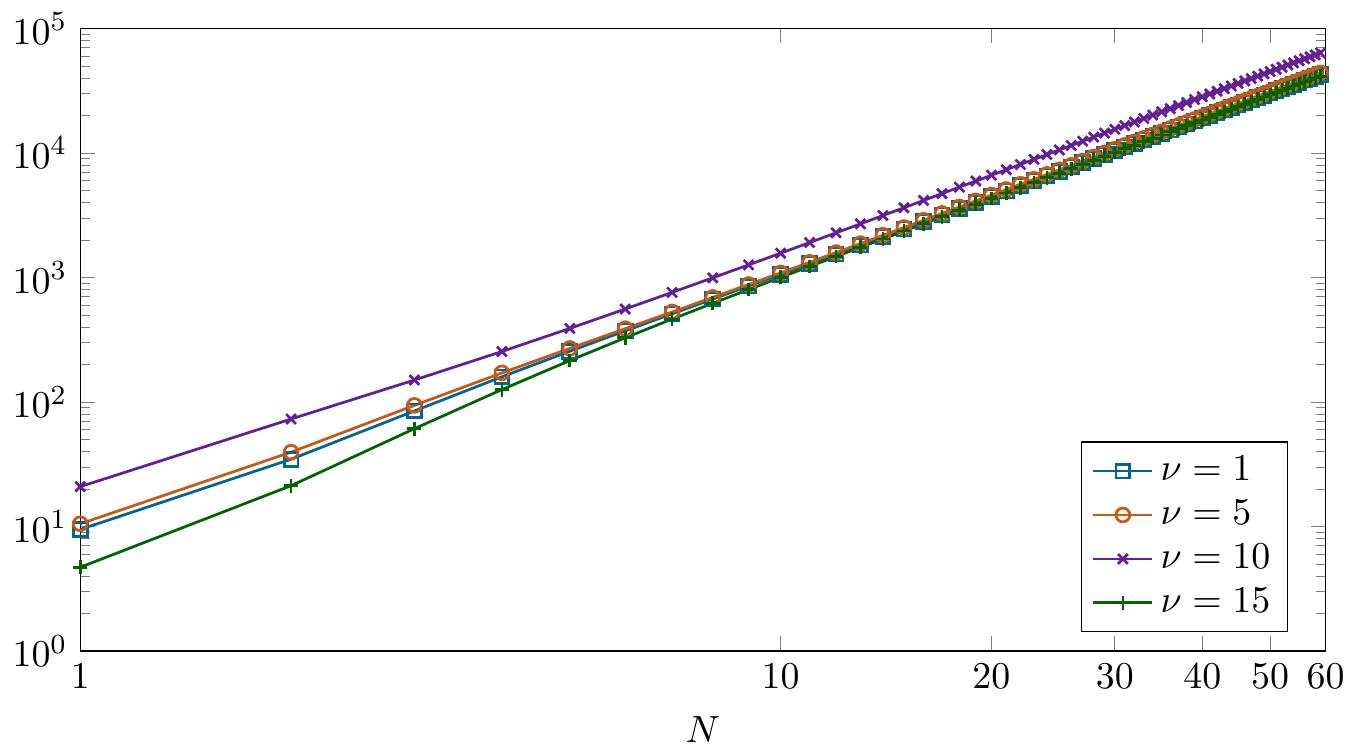}
  \caption{Condition numbers of the discretization matrices of $\tilde s^\sigma+\lambda_\nu\tilde m_0^\sigma-\omega^2\tilde m_1^\sigma$ with parameters $\sigma=0.3+0.3i$, $\omega=10-0.5i$, $R=1$}
  \label{fig:condition_numbers}
\end{figure}

Overall, the numerical results agree perfectly with the theoretical ones in situations, where the analytical solution is known. Note, that we have not chosen optimal parameters for the numerical tests. A kind of optimal complex scaling would be the choice $\sigma=\iu/\omega$ (see Subsection \ref{sec:BestAppr1d} and \cite{nw2018}). 

The best choice of the interface (the parameter $R$) is not so obvious. Small values of $R$ reduce the rate of convergence of the super algebraic part of the error. On the other hand, the computational costs increase for larger interior domains (i.e. for large values of $R$). Moreover, discretization errors in the interior domain become more and more dominant if the interior domain is large, since resonance functions are typically exponentially increasing in the interior domain (cf. the results in \cite{nw2018}). Hence, we propose to choose rather  small interior domains and optimize the approximation error in the exterior domain by choosing appropriate parameters $\sigma$ and $N$. 
\begin{figure}[H]
  \centering
  \captionsetup[subfigure]{justification=centering}
  \begin{subfigure}[t]{0.48\textwidth}
    \centering
  \includegraphics{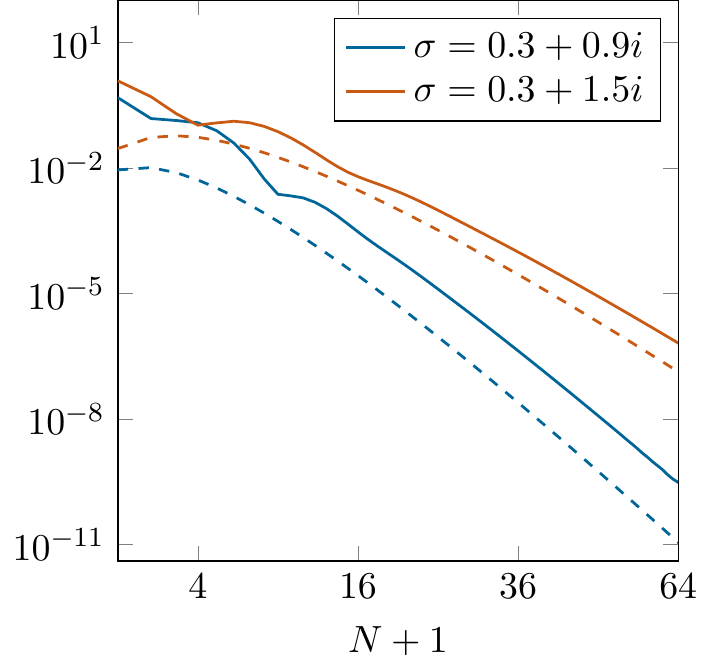}
  \caption{Super algebraic convergence for $R=1$ and varying $\sigma$.}
  \label{fig:convergence_sqrt}
  \end{subfigure}
  \begin{subfigure}[t]{0.48\textwidth}
  \includegraphics{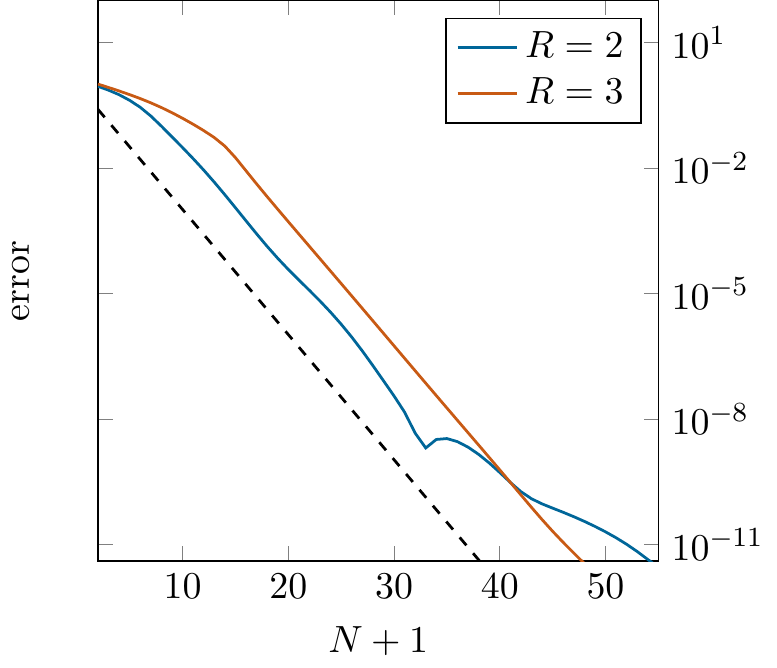}
    \centering
    \caption{Exponential convergence for $\sigma=0.3+1.5i$ and varying $R$.}
  \label{fig:convergence_r}
  \end{subfigure}
  \caption{Errors of the eigenvalue $\omega\approx2.903916-1.201866i$ obtained by solving the separated problem for $\nu=3$. The dashed lines mark the squared exponential and super algebraic convergence rates from Section \ref{sec:convergence} respectively. Note the different scalings on the horizontal axes.}
  \label{fig:convergence}
\end{figure}

%%%%%%%%%%%%%%%%%%%%%%%%%%%%%%%%%%%%%%%
\subsection{Computational costs}
\label{sec:performance}
In this subsection we compare the computational costs of our infinite elements and a conventional PML by approximating the resonances of the Helmholtz equation on $$\Oe:=\domain:=\setR^3\setminus B_1(0)=\left\{\varx\in\setR^3:\|\varx\|>1\right\}.$$

All computations in this section were done on a desktop computer with an Intel i3 CPU with 2x3.5GHz and 16GiB memory. The eigenvalues were calculated using a shift-and-invert Arnoldi algorithm (cf. \cite{Saad}) and a direct inverse via a Cholesky factorization for complex symmetric matrices. All given times are for the factorization of the given system matrix only, since this is the main contribution to the overall computational costs.

Figure \ref{fig:perf} shows the error plotted against factorization times for infinite elements and a PML using the same tensor product method described in Section \ref{sec:tp_ext_disc} but with one dimensional high order finite element basis functions in radial direction on an interval $[0,T]$. We applied $h$-refinement to obtain a succession of discretizations.
In Figures \ref{fig:perf_T_5_i_0} and \ref{fig:perf_T_5_i_9} the error generated by the truncation of the exterior domain can be observed at approximately $10^{-3}$. In Figures \ref{fig:perf_T_5_i_9} and \ref{fig:perf_T_8_i_9} the infinite elements already reach the error generated by the surface discretization which is approximately $10^{-7}$. All experiments show that the infinite elements are clearly superior to the used PML discretizations with respect to computational efficiency. Note, that due to the fact that we used the tensor product ansatz also for the PML discretizations this version of PML is already more efficient than a typical PML based on an unstructured exterior mesh.

The largest problems in the examples above had $99425$ degrees of freedom ($24$ in radial direction). Factorizing the inverse took up about 7GiB of memory.

\begin{figure}[H]
\centering
  \begin{subfigure}[t]{0.45\textwidth}
    \centering
  \includegraphics{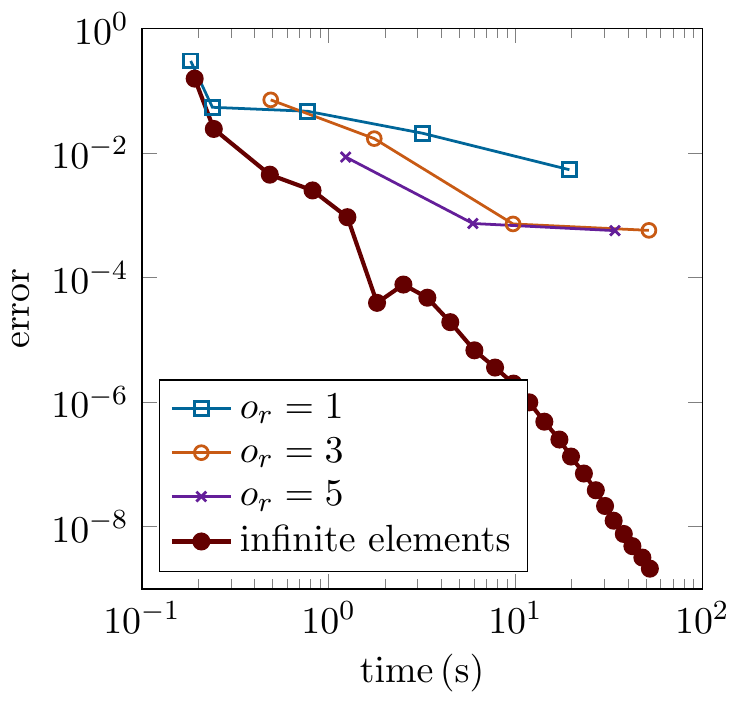}
  \caption{$T=5$,\\ $\omega\approx2.90391653245-1.20186645975i$}
  \label{fig:perf_T_5_i_0}
  \end{subfigure}
  \begin{subfigure}[t]{0.45\textwidth}
    \centering
  \includegraphics{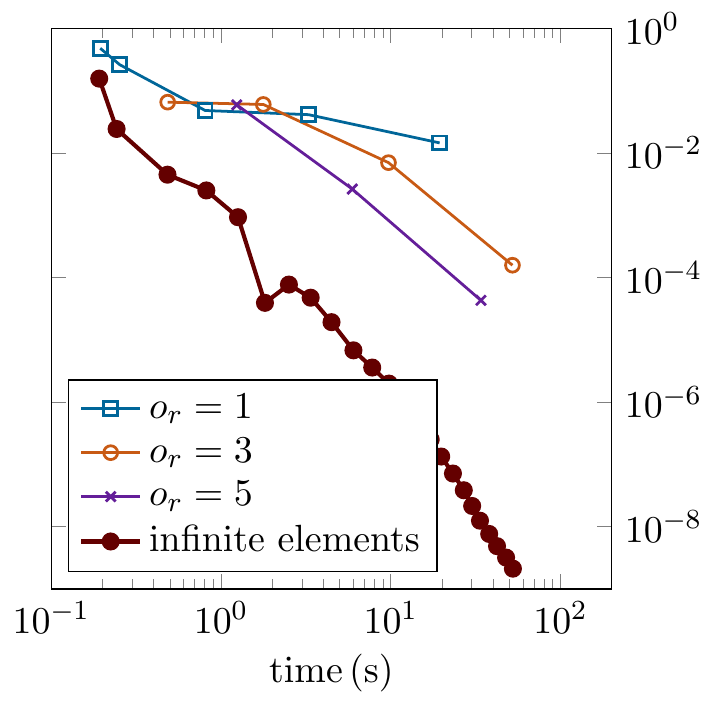}
  \caption{$T=8$,\\ $\omega\approx2.90391653245-1.20186645975i$}
  \label{fig:perf_T_8_i_0}
  \end{subfigure}\\
  \begin{subfigure}[t]{0.45\textwidth}
    \centering
  \includegraphics{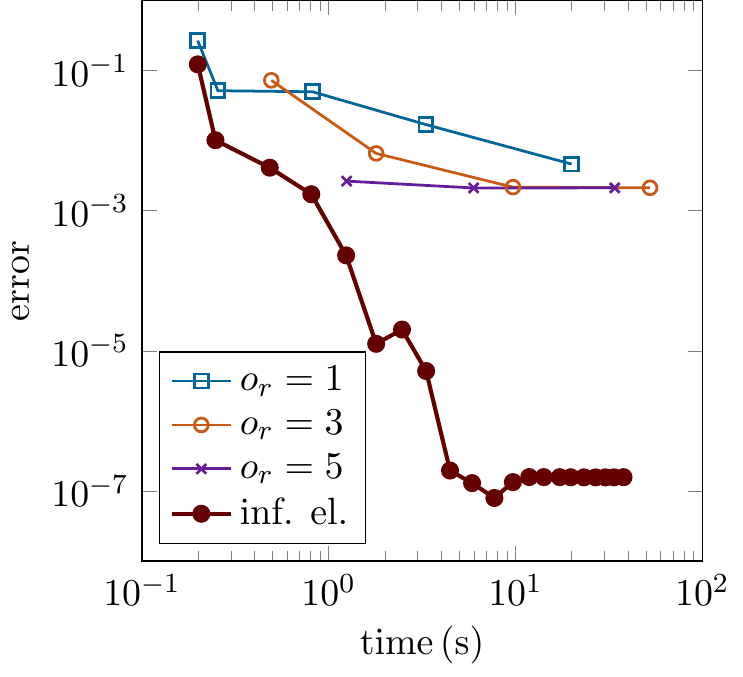}
  \caption{$T=5$,\\ $\omega\approx5.77658328455	-1.41788771722i$}
  \label{fig:perf_T_5_i_9}
  \end{subfigure}
  \begin{subfigure}[t]{0.45\textwidth}
    \centering
  \includegraphics{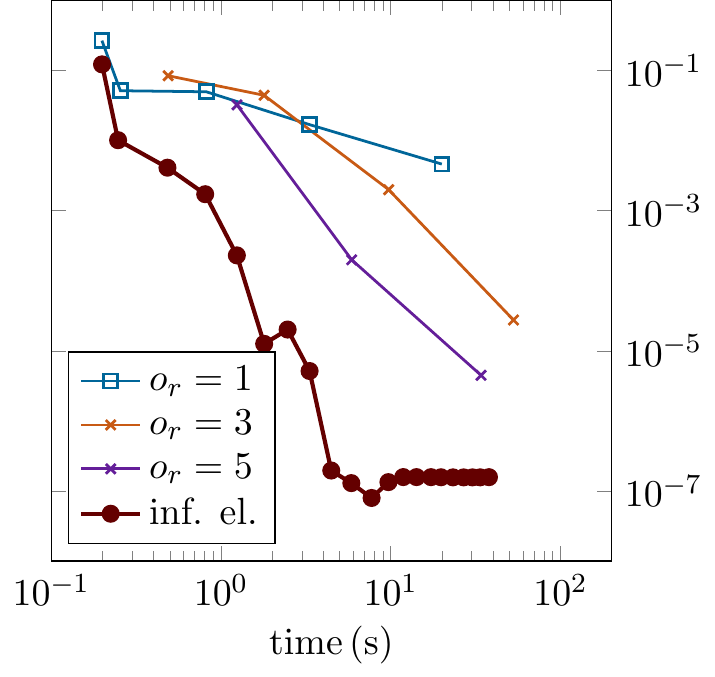}
  \caption{$T=8$,\\ $\omega\approx5.77658328455	-1.41788771722i$}
  \label{fig:perf_T_8_i_9}
  \end{subfigure}
  \caption{Comparison of errors against factorization times for eigenvalues $\omega$ and a surface discretization with order $o=5$ and meshsize $h=0.3$, $\sigma=\frac{1+i}{\omega}$ and different exterior discretizations. We use PMLs with truncation at $T$ and radial elements of order $o_r$.}
  \label{fig:perf}
\end{figure}

\subsection{An example with inhomogeneous exterior}
\label{sec:example}

In this subsection we approximate the resonances of the Helmholtz equation on $\Oe:=\domain:=\setR^3\setminus B_1(0)=\left\{\varx\in\setR^3:\|\varx\|>1\right\}$ for a potential 
$$p\of{\varx\of{\radvar,\bdvar}}:=(1+\hat\epsilon \hat p\of\bdvar)(1+\tilde\epsilon\tilde p\of\radvar),$$
with functions $\hat p:\bdi\to\setR$, $\tilde p:\setR_{\geq 0}\to\setR$. Note, that such an example is not covered by our approximation results of Section \ref{sec:convergence}. Nevertheless since our infinite elements allow the use of numerical integration the application of infinite elements to examples with inhomogeneous exterior suggests itself.

If we assume $\hat\epsilon$ to be zero, the equation can be separated using an ansatz
$$\sol\of{\varx\of{\radvar,\bdvar}}:=\radsol_{\nu}\of\radvar Y_{\nu,j}\of\bdvar.$$
Since the spherical harmonics $Y_{\nu,j}$ are eigenfunctions of the surface Laplacian and the according bilinear form $\hat s$ with the corresponding eigenvalues $\nu(\nu+1)$, this ansatz, combined with complex scaling as before leads to the set of one dimensional eigenvalue problems
$$\tilde s^\sigma\of{\radsol,\radtestf}+\nu(\nu+1)\tilde m_0^\sigma\of{\radsol,\radtestf}=\omega^2\tilde m_1^\sigma\of{(1+\tilde\epsilon)\tilde p\radsol,\radtestf},\quad \nu\in\setN_0.$$
All three dimensional experiments use the parameters $N=50$, finite element mesh size $h=0.3$ and polynomial order $p=4$. For the one dimensional examples we used $N=200$.

\begin{figure}[H]
  \centering
  \captionsetup[subfigure]{justification=centering}
  \begin{subfigure}[t]{0.3\textwidth}
    \includegraphics{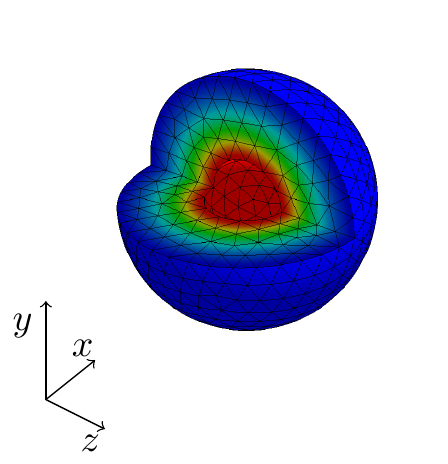}
    \caption{$1+1.5\radmod p$,\\ with color coding from $1$ (blue) to $2$ (red).}
    \label{fig:pottilde}
  \end{subfigure}
  \begin{subfigure}[t]{0.3\textwidth}
    \includegraphics{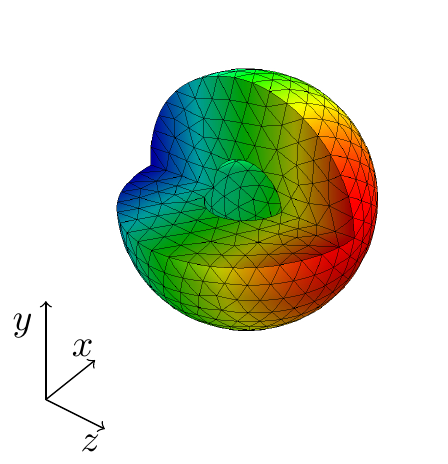}
    \caption{$1+0.5\bdmod p$,\\ with color coding from $0.5$ (blue) to $1.5$ (red).}
    \label{fig:pothat}
  \end{subfigure}
  \begin{subfigure}[t]{0.3\textwidth}
    \includegraphics{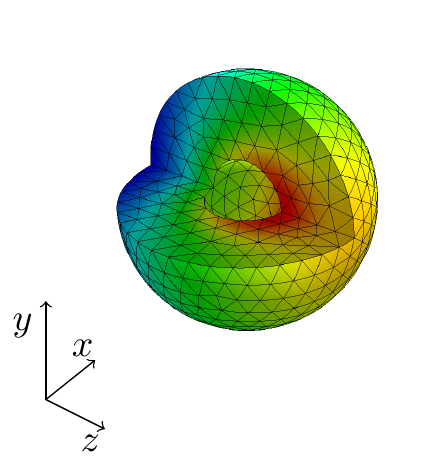}
    \caption{$(1+1.5\radmod p)(1+0.5\bdmod p)$,\\ with color coding from $0.5$ (blue) to $2$ (red).}
    \label{fig:pot}
  \end{subfigure}
  \caption{Potential functions}
  \label{fig:pots}
\end{figure}
Figure \ref{fig:hh_r_perturbed} shows the eigenvalues of the discretized separated problem for 
$$\tilde p\of\radvar:=\frac{(\radvar-1)^2}{1+(\radvar-1)^4},$$
(cf. Figure \ref{fig:pottilde}), $\nu=0,\ldots,5$ and different choices of $\tilde\epsilon$, as well as the eigenvalues of the full three dimensional simulation. For $\radmod\epsilon=0$ the eigenvalues can be calculated exactly by finding the roots of $h_\nu'$, the derivative of the spherical Hankel function of first kind of order $\nu$. For larger values of $\tilde\epsilon$ the eigenvalues move closer to the real axis. The approximated eigenvalues of the full three dimensional simulation show a good agreement with the ones of the separated problem. The resonances located close to the negative imaginary axis in Figures \ref{fig:hh_r_perturbed} and \ref{fig:hh_perturbed} are part of the discretization of the essential spectrum (cf. \cite{ColMonk,nw2018,KimPasciak:09}).

\begin{figure}[H]
  \centering
  \includegraphics[width=0.8\textwidth]{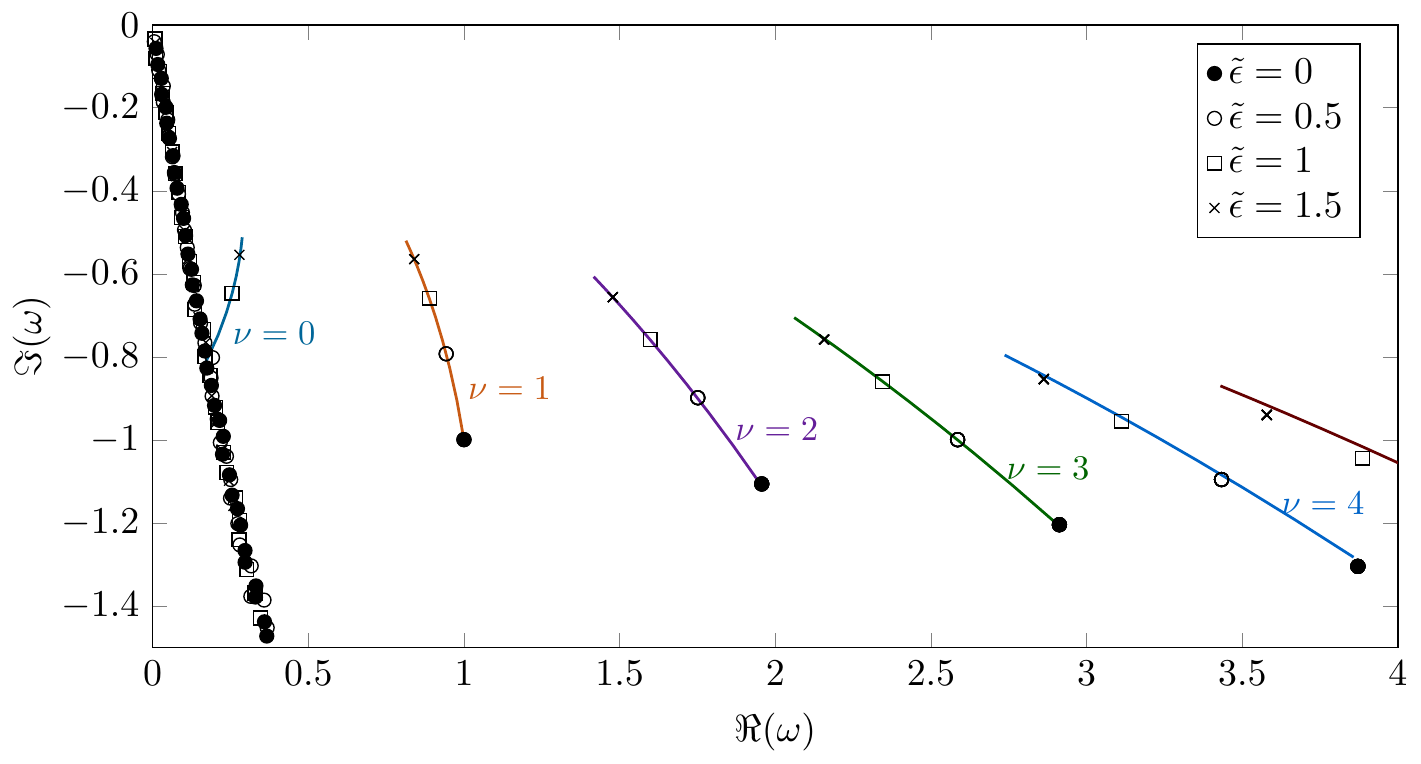}
  \caption{Inhomogeneous exterior problem with radial inhomogeneity. The lines mark the locations of resonances for a given Hankel index $\nu$ and varying $\tilde\epsilon$, obtained by solving the separated problem whereas the marks are eigenvalues calculated by solving the full three dimensional problem.}
  \label{fig:hh_r_perturbed}
\end{figure}

Figure \ref{fig:hh_perturbed} shows resonances of the same problem with an additional potential
  $$\bdmod p\of{\bdvar}:=\hat p\of{x,y,z}:=z,$$
(cf. Figures \ref{fig:pothat} and \ref{fig:pot}) and varying values of $\hat\epsilon$.  This problem is not separable any more, thus only a three dimensional simulation is possible. Due to the disturbed symmetry, the multiple eigenvalues fan out. 

\begin{figure}[H]
\centering
  \begin{subfigure}{0.45\textwidth}
  \includegraphics{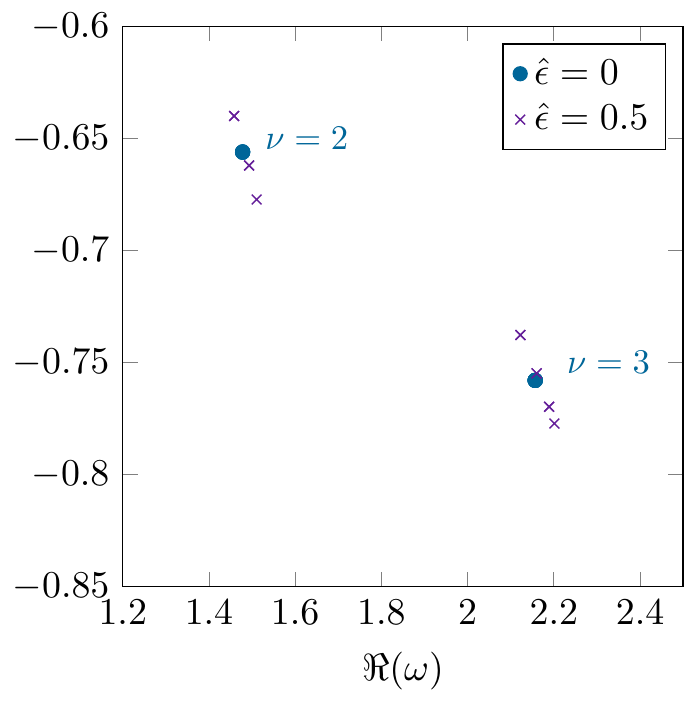}
  \end{subfigure}
  \begin{subfigure}{0.45\textwidth}
  \includegraphics{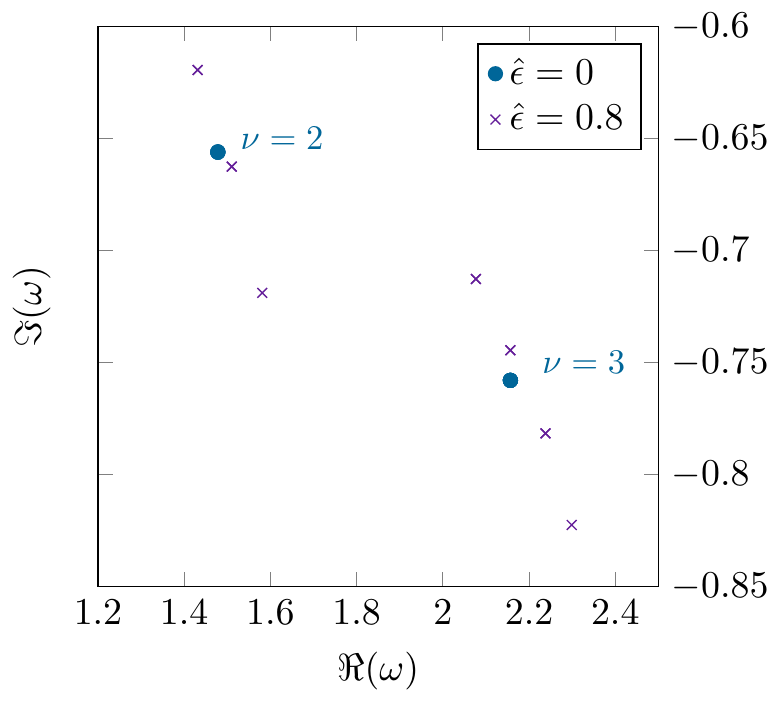}
  \end{subfigure}
  \caption{Inhomogeneous exterior problem with $\tilde\epsilon=1.5$ and variable $\hat\epsilon$. Although this is a three dimensional simulation, the given indices $\nu$ refer to the corresponding Hankel indices for the separated problem for $\hat\epsilon=0$}
  \label{fig:hh_perturbed}
\end{figure}
Figures \ref{fig:efs_r_pert} and \ref{fig:efs_pert} show selected resonance functions corresponding to the previously approximated resonances. To visualize the resonance functions $\Oi=B_2(0)\setminus B_1(0)$ was chosen here. In Fig.~\ref{fig:efs_r_pert} $\bdmod\varepsilon$ is zero, i.e. we have a rotationally invariant problem with radial inhomogeneity. Hence, the resonance functions are rotationally invariant as well. In Fig.~\ref{fig:efs_pert} the problem is not rotational invariant leading to resonance functions with perturbed symmetry. 
\begin{figure}[H]
\centering
  \begin{subfigure}{0.45\textwidth}
    \includegraphics{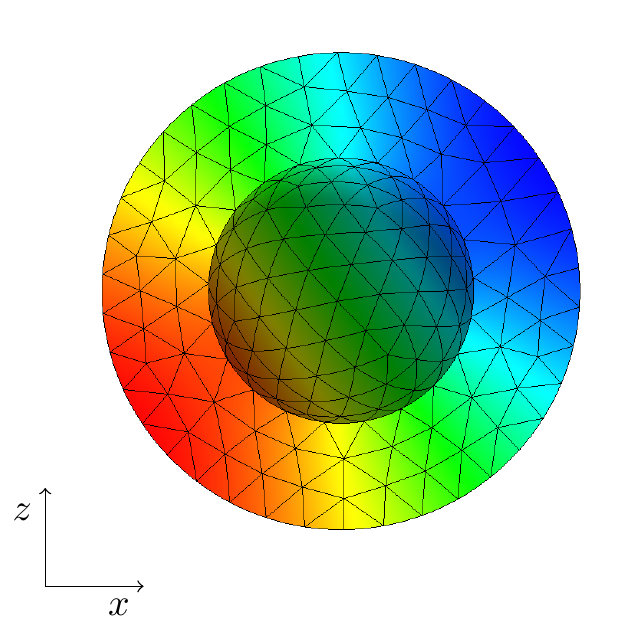}
    \caption{$\omega\approx0.87-0.50i$, $\nu=1$}
\end{subfigure}
  \begin{subfigure}{0.45\textwidth}
    \includegraphics{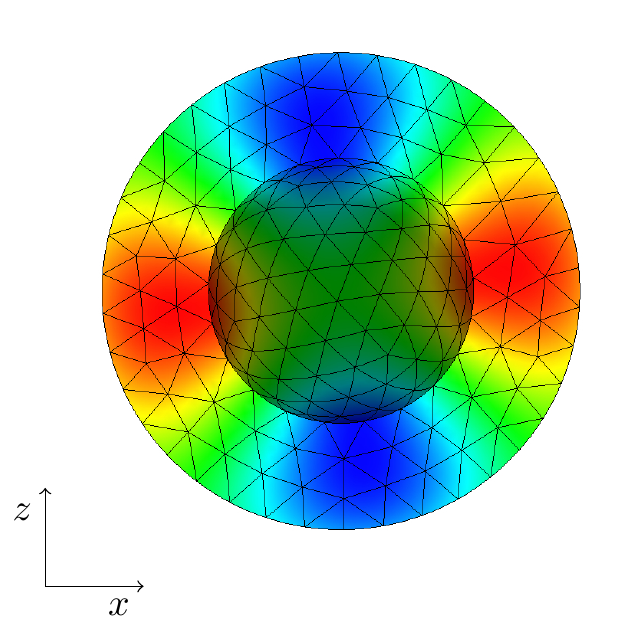}
    \caption{$\omega\approx1.47-0.52i$, $\nu=2$}
  \end{subfigure}
  \caption{Resonance functions corresponding to eigenvalues from Figure \ref{fig:hh_r_perturbed} with $\tilde\varepsilon=0.5$ and $\bdmod\varepsilon=0$. The given indices $\nu$ refer to the corresponding Hankel indices in the separated case.}
  \label{fig:efs_r_pert}
\end{figure}

\begin{figure}[H]
\centering
  \begin{subfigure}{0.45\textwidth}
    \includegraphics{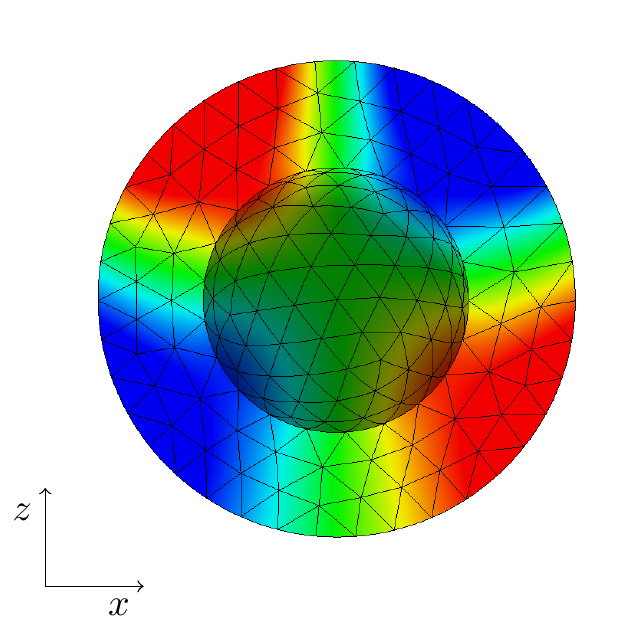}
    \caption{$\omega\approx1.49-0.52i$}
\end{subfigure}
  \begin{subfigure}{0.45\textwidth}
    \includegraphics{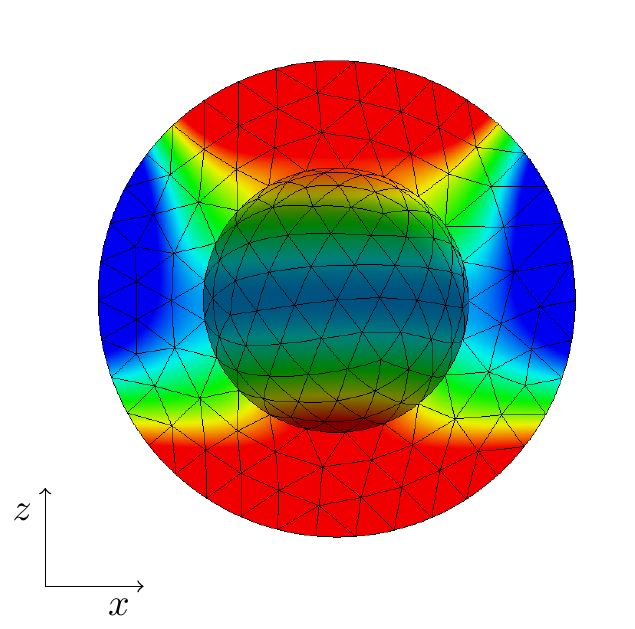}
    \caption{$\omega\approx1.50-0.53i$}
  \end{subfigure}
  \caption{Resonance functions corresponding to eigenvalues from Figure \ref{fig:hh_perturbed} with $\hat\varepsilon=0.5$ and $\bdmod\varepsilon=1.5$.}
  \label{fig:efs_pert}
\end{figure}

\section{Conclusion}
\label{sec:conclusion}

In summary we can say, that complex scaled infinite elements are a very effective method for Helmholtz resonance problems. In comparison to PMLs they are more efficient due to their super algebraic convergence and have the advantage of less method parameters to tune. Our approximation results give guidance on how to choose said parameters.
Our numerical results suggest, that complex scaled infinite elements are also applicable to inhomogeneous exterior domains.
A straightforward extension to the presented method would be the application of a frequency dependent complex scaling or the use of cartesian or normal scaling directions.

\newpage
\bibliographystyle{plain} 
\bibliography{bibliography}
\end{document}